\documentclass[11pt,letterpaper]{amsart}

\usepackage{amsmath,amssymb,amsthm,amsxtra,mathrsfs}
\usepackage{yhmath}
\usepackage{graphicx}
\usepackage{cancel,ulem}
\usepackage{color}
\usepackage[all,cmtip]{xy}
\usepackage{hyperref}
\usepackage{enumitem}
\usepackage{color}
\usepackage{bm}
\usepackage{subcaption}
\usepackage{tikz}
\usetikzlibrary{arrows}
\usepackage{accents}
\usepackage{bbm}
\usepackage{dsfont}

\newcommand{\jac}{\text{Jac}}
\newcommand{\vertiii}[1]{{\left\vert\kern-0.25ex\left\vert\kern-0.25ex\left\vert #1 
    \right\vert\kern-0.25ex\right\vert\kern-0.25ex\right\vert}}

\def\r{{\rangle}}
\def\l{{\langle}}
\def\S{{\mathbb S}}

\def\R{{\mathbb R}}

\newtheorem{theorem}{Theorem}
\newtheorem{lemma}[theorem]{Lemma}

\newtheorem{proposition}[theorem]{Proposition}

\theoremstyle{definition}

\theoremstyle{remark}
\newtheorem{remark}[theorem]{Remark}

\numberwithin{equation}{section}
\numberwithin{theorem}{section}
\numberwithin{problem}{section}

\begin{document}

\begin{abstract}
In this paper, we give a unified treatment of the local well-posedness for the wave kinetic equation in almost critical weighted $L^r$ spaces with $2 \leq r \leq \infty.$ The proof builds on ideas from our earlier works \cite{AmLe24, AmLemain25}. Our approach is based solely on kinetic tools, with no appeal to Fourier theory. 

\end{abstract}

\title[On the optimal local well-posedness of the WKE in $L^r$]{On the optimal local well-posedness of the wave kinetic equation in $L^r$}

\author[I. Ampatzoglou]{Ioakeim Ampatzoglou}
\address{Baruch College \& The Graduate Center, City University of New York, Newman Vertical Campus, 55 Lexington Ave, New York, NY, 10010, USA}
\email{ioakeim.ampatzoglou@baruch.cuny.edu ; iampatzoglou@gc.cuny.edu}

\author[T. L\'eger ]{Tristan L\'eger}
\address{Yale University, Mathematics Department, Kline tower, New Haven, CT 06511, USA}
\email{tristan.leger@yale.edu}

\maketitle
\tableofcontents

\section{Introduction}

\subsection{Background}
In this paper we study the 3D  wave kinetic equation (WKE) with Laplacian dispersion relation
\begin{equation} \label{KWE}\tag{WKE}
\begin{cases}
\partial_t f = \mathcal{C}[f]\\
f(t=0)=f_0
\end{cases}
\end{equation}
where $f:[0,T] \times  \mathbb{R}^3 \rightarrow \mathbb{R}$, $T>0$ and $f_0:\R^3\to\R$.

The collisional operator $\mathcal{C}$ is defined as 
\begin{align} \label{collision}
\begin{split}
    \mathcal{C}[f] & := \int_{\mathbb{R}^{9}} \delta (\Sigma) \, \delta(\Omega) \, f f_1 f_2 f_3 \, \big(\frac{1}{f} + \frac{1}{f_1} - \frac{1}{f_2} - \frac{1}{f_3} \big) \,dk_1 dk_2 dk_3,  \\
    \Sigma&:=k+k_1 - k_2 - k_3, \\
    \Omega &:= \vert k \vert^2 + \vert k_1 \vert^2 - \vert k_2 \vert^2 - \vert k_3 \vert^2 .
\end{split}
\end{align}
This describes the statistical properties of a system governed by the cubic NLS equation, and is the most canonical model in the field of weak turbulence. Here the unknown $f$ corresponds to the two-point correlation function of the system. More generally, the goal of this theory is to describe the out-of-equilibrium dynamics of interacting waves. It originated in 1929 with the work of R. Peierls \cite{Pe} and in 1962 with K. Hasselman \cite{ha62,ha63} independently. It has proved very versatile and has been applied to many physical systems. For a comprehensive list of concrete examples, we refer to the textbook of S. Nazarenko \cite{Nazarenko}.

This theory has also been studied mathematically, with the question of derivation of kinetic equations from microscopic models as the main focus thus far. After works by several groups \cite{BGHS, CG1, CG2, DH1}, the optimal result was obtained by Z. Hani and Y. Deng \cite{DH2}. They derived the wave kinetic equation from the microscopic system on the optimal time scale, called the kinetic time.


We note that these results are all conditional on the wave kinetic equation having a good well-posedness theory. Indeed only smooth enough solutions to the kinetic equation can be derived from the microscopic system. The next natural step is thus to develop the local well-posedness theory of the WKE, which is the subject of the present paper. We note that weak solutions were constructed by M. Escobedo and J. Vel\'{a}zquez in \cite{EV}. We will focus on strong solutions here, since they correspond to configurations that have been shown to be derivable from many-body systems.

The well-posedness theory of strong solutions for wave kinetic equations is still in its infancy. A particularly relevant functional analytic setting to consider are scale-critical weighted Lebesgue spaces. Indeed, they contain Rayleigh-Jeans spectra given by $f(k) = \frac{1}{\mu + \vert k \vert^2},$ where $\mu$ is a free parameter. These special solutions correspond to thermodynamic equilibrium and are thus the analog of Maxwellians for the Boltzmann equation. As such, they are important to understand the long-time behavior of \eqref{KWE}. We note here that kinetic equations have other physically relevant solutions, e.g. the celebrated Kolmogorov-Zakharov spectra. The full dynamics of the equation is therefore expected to be quite complicated and different from that of the Boltzmann equation. Local results point in this direction as well: C. Collot, H. Dietert and P. Germain \cite{CollotDietertGermain24} studied the static problem near one Kolmogorov-Zakharov solution in the isotropic case, and prove that it is stable.  However, some Rayleigh-Jeans solutions have been shown to be unstable by M. Escobedo and A. Menegaki \cite{EscobedoMenegaki24}. 

Restricting our attention to local well-posedness, the state of the art result was obtained by P. Germain, A. Ionescu and M.-B. Tran \cite{GeIoTr} who prove almost critical local well-posedness in $L^2$ and $L^\infty.$ The $L^\infty$ bound is proved by direct estimation, while for $L^2$ the authors rely on Radon transform type techniques. 

In the present paper we improve on this result in two respects: first, we extend it to all almost critical weighted $L^r$ for $2 \leq r \leq \infty$. Second, we give a unified treatment of all the cases, and do not rely on multilinear interpolation to obtain the result. Indeed, our method of proof is quite different from \cite{GeIoTr}, and is inspired by our recent works on the wave kinetic and Boltzmann equations \cite{AmLe24, AmLemain25}. We rely solely on classical kinetic tools, such as Bobylev variables and angular averaging estimates. In particular, we make no appeal to Fourier theory.


\subsection{Results obtained}
We start by stating the main result of the paper.
\begin{theorem} \label{main-thm}
Let $r\geq 2$ and $0<\delta<1/r$ if $r<\infty$, $\delta > 0$ if $r = \infty.$ Let $f_0 \in \l k \r^{-2 + \frac{3}{r} -\delta} L^r.$ Then there exists $T=T(\|\l k\r^{2 - \frac{3}{r} + \delta} f_0\|_{L^r})>0$ such that the initial value problem \eqref{KWE}
has a unique strong solution $f(t) \in \mathcal{C} \big( [0,T] ; \l k \r^{-2 + \frac{3}{r} - \delta} L^r \big),$ that is
\begin{align*}
    \forall t \in [0,T], \quad f(t) = f_0 + \int_0^t  \mathcal{C}[f](s)\, ds .
\end{align*}

Moreover the solution depends continuously on the initial data: let $f_0, g_0 \in \l k \r^{-2 + \frac{3}{r} -\delta} L^r$, $T_1,T_2>0$ be the times of existence obtained above and $f,g$ be the corresponding solutions of \eqref{KWE} in $[0,T_1]$ and $[0,T_2]$ respectively. Then, there holds the estimate 
\begin{align}\label{continuity wrt data estimate}
     \sup_{t\in[0,T_{min}]}\big \Vert \l k \r^{2 - \frac{3}{r} + \delta} \big( f(t) - g(t) \big) \big \Vert_{L^r} \leq 2\big \Vert \l k \r^{2 - \frac{3}{r} + \delta} \big(f_0 - g_0 \big) \big \Vert_{L^r},
\end{align}
where $T_{min}:=\min\{T_1,T_2\}$.

Finally, the flow preserves positivity, in the sense that if $f_0 \geq 0,$ then for all $t \in [0,T], f(t) \geq 0,$ where $T$ is the time of existence corresponding to $f_0$.
\end{theorem}
\begin{remark}
The case $\delta = 0$ corresponds to the critical exponent. 
\end{remark}
\begin{remark}
In the special cases $r = 2, r = \infty,$ we recover the results of \cite{GeIoTr}. Of course, it is possible to obtain the intermediate cases from their bounds by multilinear interpolation.
\end{remark}
\begin{remark}
As will be clear in the proof, all the cases $2 \leq r \leq \infty$ are treated in a unified way.
\end{remark}

\begin{remark}
The proof yields a more quantitative estimate for the lifespan of the solution. More precisely, it shows that $T \gtrsim \|\l k\r^{2 - \frac{3}{r} + \delta} f_0\|_{L^r}^{-2}$.
\end{remark}

\color{black}

\subsection{Parametrization of the resonant manifolds} 
For Laplacian dispersion relation, there is a particularly simple parametrization of the collisional kernel as a hard-sphere quantum Boltzmann-type operator. In other words, the wave interaction of the modes $k,k_1,k_2,k_3$ can be identified with the elastic collision of two particles with pre-collisional velocities $k,k_1$ and post-collisional velocities $k_2,k_3$. At the level of the corresponding particle interaction, the resonant conditions $\Sigma=0$ and $\Omega=0$ correspond to the conservation of momentum and energy respectively.

More precisely, consider $F:\R^{12}\to\R$   continuously differentiable and compactly supported, \color{black} and denote
$$I_F(k):=\int_{\R^{9}}\delta(\Sigma)\delta(\Omega)F(k,k_1,k_2,k_3)\,dk_1\,dk_2\,dk_3.$$

Recalling  $\Sigma=k+k_1-k_2-k_3$, $\Omega=|k|^2+|k_1|^2-|k_2|^2-|k_3|^2$, and using Fubini's theorem and the co-area formula, we can write
\begin{align*}
I_F(k)&=\int_{\R^3}\left(\int_{\R^3}\delta(\Omega_{k,k_1}(k_2))F(k,k_1,k_2,k+k_1-k_2)\,dk_2\right)\,dk_1\\
&=\int_{\R^3}\int_{\Omega_{k,k_1}=0}\frac{1}{|\nabla_{k_2}\Omega_{k,k_1}(k_2)|}F(k,k_1,k_2,k+k_1-k_2)\,d\sigma_{k,k_1}(k_2)\,dk_1,
\end{align*}
where 
\begin{align*}
\Omega_{k,k_1}(k_2)&=|k_2|^2+|k+k_1-k_2|^2-|k|^2-|k_1|^2=2\left(|k_2-K|^2-\frac{|w|^2}{4}\right),\\
K&:=\frac{k+k_1}{2},\quad w:=k-k_1,
\end{align*}
and $\,d\sigma_{k,k_1}$ denotes the surface measure on the surface $\Omega_{k,k_1}=0$. As a result, the surface $\Omega_{k,k_1}=0$ can be parametrized by 
\begin{equation*}
k_2=K-\frac{|w|}{2}\sigma,\quad \sigma\in\S^2,\quad d\sigma_{k,k_1}(k_2)=\frac{|w|^2}{4}\,d\sigma.   
\end{equation*}
Finally on the surface $\Omega_{k,k_1}=0$, we readily compute
\begin{align*}
|\nabla_{k_2}\Omega_{k,k_1}(k_2)|=4|k_2-K|=2|w|.    
\end{align*}
Combining these computations, we conclude that

\begin{equation}\label{IF parametrized}
I_F(k)=\frac{1}{8}\int_{\R^3\times\S^2}|w| F(k,k_1,k^*,k_1^*)\,d\sigma\,dk_1,
\end{equation}
where 
\begin{equation}\label{collisional law}
\begin{cases}
k^*=K-\frac{|w|}{2}\sigma\\
k_1^*=K+\frac{|w|}{2}\sigma
\end{cases}  \quad  ,
\quad K=\frac{k+k_1}{2},\quad w=k-k_1.
\end{equation}
Due to the linear growth $|w|$ in the integrand, equation \eqref{IF parametrized} essentially shows that integration over the resonant  manifolds is equivalent to a collisional integral of interacting hard-spheres. The resulting linear growth $|w|$ is called resonant cross-section. 

This idea was first implemented in the study of wave turbulence with Laplacian dispersion relation by the first author of this paper \cite{Am}, where global existence, uniqueness and stability of mild solutions to the space inhomogeneous WKE were proved for sufficiently small exponentially decaying initial data. It was later extended in \cite{AmMiPaTa24} to small polynomially decaying initial data and to the corresponding hierarchy of equations by the first author, J.K. Miller, N. Pavlovi\'c and M. Taskovi\'c. In \cite{AmLe24}, we improved this result to translation invariant spaces in the spatial variable by introducing collisional averaging estimates combined with dispersive properties of free transport. We also obtained a full description of the asymptotic behavior of the solutions, showing that they scatter. In \cite{AmLeill25} we used similar collisional averaging estimates to show strong local well-posedness of MMT-type kinetic equations below a sharp ill-posedness threshold. Above said threshold the collisional averaging effect is absent and we showed that, coincidentally, the equations are ill-posed. Thus these collisional averaging estimates characterize completely the well-posedness of these equations. These techniques are quite versatile, and can be applied to other models. We refer for example to the work of N. Pavlovi\'c, M. Taskovi\'c and L. Velasco \cite{PaTaVe25} who constructed mild solutions that scatter for the six-wave WKE with exponentially decaying initial data.

\subsection{Collisional operators}
Using \eqref{IF parametrized}, the collisional operator can be written in gain and loss form as follows 
\begin{equation}\label{gain-loss}
\mathcal{C}[f]=\mathcal{Q}^+[f]-\mathcal{Q}^-[f],    
\end{equation}
where 
\begin{align}
\mathcal{Q^+}[f](k)&=\frac{1}{8}\int_{\R^3\times\S^2}|w|f(k^*)f(k_1^*)\big(f(k)+f(k_1)\big)\,d\sigma\,d  k \color{black} _1\label{first def of Q+},\\
\mathcal{Q}^-[f](k)&=f(k)\mathcal{R}[f](k),\label{first def of Q-}\\
\mathcal{R}[f](k)&= \frac{1}{8}\int_{\R^3\times\S^2}|w|f(k_1)\big(f(k^*)+f(k_1^*)\big)\,d\sigma\,dk_1.
\end{align}
The operator $\mathcal{R}[f]$ is called collision frequency.

Now, we use a more refined decomposition which takes advantage of the symmetry of the collisional operators. This idea was introduced in \cite{AmLe24} and then was also used in \cite{AmLeill25}. For this, let us introduce the standard cut-off function $\chi:=\mathds{1}_{(0,+\infty)}$. Noticing that the substitution $\sigma\to -\sigma$ just exchanges $k^*$ with $k_1^*$, we obtain,  denoting $w:= k-k_1$ and using the standard notation $\widehat{w} = \frac{w}{\vert w \vert}$ \color{black}
\begin{align*}
\mathcal{Q}^+[f](k)&=  \frac{1}{8}\int_{\R^3\times\S^2}|w|f(k^*)f(k_1^*)\big(f(k)+f(k_1)\big)\chi(\widehat{w}\cdot\sigma)\,d\sigma\,d  k \color{black} _1 \\
&\hspace{2cm}+ \frac{1}{8}\int_{\R^3\times\S^2}|w|f(k^*)f(k_1^*)\big(f(k)+f(k_1)\big)\chi(-\widehat{w}\cdot\sigma)\,d\sigma\,d  k \color{black} _1\\
&=\frac{1}{4}\int_{\R^3\times\S^2}|w|f(k^*)f(k_1^*)\big(f(k)+f(k_1)\big)\chi(\widehat{w}\cdot\sigma)\,d\sigma\,d  k \color{black} _1,
\end{align*}
where for the last line we used the substitution $\sigma\to-\sigma$ in the second part of the sum.

Similarly, we obtain
\begin{align*}
\mathcal{R}[f](k)&= \frac{1}{8}\int_{\R^3\times\S^2}|w|f(k_1)\big(f(k^*)+f(k_1^*)\big)\chi(\widehat{w}\cdot\sigma)\,d\sigma\,dk_1\\
&\hspace{2cm} +\frac{1}{8}\int_{\R^3\times\S^2}|w|f(k_1)\big(f(k^*)+f(k_1^*)\big)\chi(-\widehat{w}\cdot\sigma)\,d\sigma\,dk_1\\
&=\frac{1}{4}\int_{\R^3\times\S^2}|w|f(k_1)\big(f(k^*)+f(k_1^*)\big)\chi(\widehat{w}\cdot\sigma)\,d\sigma\,dk_1.
\end{align*}
We conclude that we can write
\begin{align}
\label{defQ+}
    \mathcal{Q}^+[f] &:= \mathcal{G}_0[f,f,f]+\mathcal{G}_1[f,f,f], \\
    \label{defQ-} \mathcal{Q}^-[f] &:= \mathcal{L}_0[f,f,f] + \mathcal{L}_1[f,f,f],\\
    \label{defR}\mathcal{R}[f]&=\mathcal{R}_0[f,f]+\mathcal{R}_1[f,f],
\end{align}
where $\mathcal{G}_0,\mathcal{G}_1,\mathcal{L}_0,\mathcal{L}_1,\mathcal{R}_0,\mathcal{R}_1$ are the cubic/quadratic restrictions of the trilinear/bilinear operators
\begin{align} 
\label{defG0} \mathcal{G}_0[f,g,h](k)&=\frac{1}{4}f(k)\int_{\R^3\times\S^2}|w| g(k^*)h(k_1^*)\chi(\widehat{w}\cdot\sigma)\,d\sigma\,dk_1,\\
\label{defG1} \mathcal{G}_1[f,g,h](k)&=\frac{1}{4}\int_{\R^3\times\S^2}|w| f(k_1)g(k^*)h(k_1^*)\chi(\widehat{w}\cdot\sigma)\,d\sigma\,dk_1,\\ 
\label{defL0} \mathcal{L}_0[f,g,h](k)&=f(k) \mathcal{R}_0[g,h](k),\\
\label{defL1} \mathcal{L}_1[f,g,h](k)&=f(k) \mathcal{R}_1[g,h](k),\\
\label{defR0}  \mathcal{R}_0[g,h](k)&=\frac{1}{4}\int_{\R^3\times\S^2}|w| g(k_1)h(k^*)\chi(\widehat{w}\cdot\sigma)\,d\sigma\,dk_1, \\
\label{defR1}  \mathcal{R}_1[g,h](k)&=\frac{1}{4}\int_{\R^3\times\S^2}|w| g(k_1)h(k_1^*)\chi(\widehat{w}\cdot\sigma)\,d\sigma\,dk_1.
\end{align}
These multilinear operators are the fundamental objects we study to prove our main result.

\subsection{Organization of the paper}

The paper is organized as follows: in Section \ref{sec:tools} we record the technical tools used in the rest of the paper, namely Bobylev variables, basic geometric identities and key collisional averaging estimates. Then we use this toolbox to prove moment preserving trilinear estimates for the gain operators (Section \ref{sec:gain}) and the loss operators (Section \ref{sec:loss}). Finally we use these bounds to prove our main result Theorem \ref{main-thm} in Section \ref{sec:pfmain}.

\subsection{Notations} Throughout the paper, we will use the following notation:
\begin{itemize}
    \item We will write $A \lesssim B$ to mean that there exists a numerical constant $C>0$ such that $A  \leq C B.$ We will write $A \approx B$ if $A \lesssim B$ and $B \lesssim A.$ 
    \item $\chi:=\mathds{1}_{(0,+\infty)}$ will denote the standard cut-off function.
    \item Given $s\in\R$ we will denote $f_s(k):=\l k\r^s f(k)$. 
\item For any $w \in \mathbb{R}^3\setminus\{0\}$ we denote $\widehat{w} = \frac{w}{\vert w \vert}.$
    \color{black}
    \item Finally, given $k,k_1\in\R^3$ and $\sigma\in\S^2$, we will denote
    \begin{align*}
      &k^*=K-\frac{|w|}{2}\sigma,\quad k_1^*=K+\frac{|w|}{2}\sigma,\\
      &K=\frac{k+k_1}{2},\quad w=k-k_1,\quad E=|k|^2+|k_1|^2.
    \end{align*}
\end{itemize}

\subsection*{Acknowledgements}
I.A. was supported by the NSF grant  DMS-2418020 and the PSC-CUNY Research Award 68653-0056.

\section{Kinetic toolbox} \label{sec:tools}
In this section, we start by recalling basic facts about Bobylev variables. Then we record some basic geometric identities that will be used routinely in the proof. Finally, we prove collisional averaging estimates that will be crucial to offset the linear growth of the cross section in later sections.

\subsection{Bobylev variables}
First, we introduce the so-called Bobylev variables \cite{Bo75, Bo88}, a classical kinetic theory tool that we will rely on to prove crucial collisional averaging estimates. The Bobylev variables appear naturally in the gain operator above and thus will often be used to provide an alternative parametrization of the collisional operator. Here, we outline their most important properties.

Given $\sigma\in\S^2$, we define the Bobylev variables as the maps $R_\sigma^+,R_\sigma^-:\R^3\to\R^3$ defined by 
\begin{align}
R_\sigma^+(y)=\frac{y}{2}+\frac{|y|}{2}\sigma,\quad R_\sigma^-(y)=\frac{y}{2}-\frac{|y|}{2}\sigma.\label{R def}
\end{align}
The following identities are easily checked:
\begin{align}
R_\sigma^+(y)+R_\sigma^-(y)&=y,\label{decomposition of u}\\
R_\sigma^+(y)\cdot R_\sigma^-(y)&=0,\label{orthogonality}\\
|R_\sigma^+(y)|^2+|R_\sigma^-(y)|^2&=|y|^2.\label{pythagorean}
\end{align} 
Furthermore, recalling that $w = k-k_1$, we obtain
\begin{align}\label{R+ R- v}
R_\sigma^+(w)=k-k^*=k_1^*-k_1,\quad R_\sigma^-(w)=k-k_1^*=k^*-k_1.
\end{align}

 The next result shows that the Bobylev variables are diffeomorphisms when restricted appropriately. It also provides useful identities relating the magnitudes and angles between these various variables. The proof  of this result is contained in our previous papers \cite{AmLe24,AmLemain25}. 

\begin{proposition} \label{chgvarkin}
Let $\sigma\in\S^{2}$ and $\epsilon \in\{+,-\}$.
Then the map
$$R_\sigma^{\epsilon}:\lbrace y\in\R^3: y \cdot \sigma \neq - \epsilon  \vert y\vert \rbrace\to \lbrace \nu\in\R^3: \epsilon \,(\nu \cdot \sigma) > 0 \rbrace,$$ is a diffeomorphism with inverse
\begin{align}\label{inverse function}
    y=\big(R_{\sigma}^{\epsilon}\big)^{-1}(\nu) = 2 \nu  -\frac{\vert \nu \vert}{ (\widehat{\nu}\cdot \sigma )} \sigma,
\end{align}
and Jacobian 
\begin{equation}\label{Jacobian}
    \jac\,(R_\sigma^{\epsilon})^{-1}(\nu)=\frac{4 }{(\widehat{\nu}\cdot\sigma)^2}.
\end{equation}    
Moreover, for any $u\in\R^3$ with $y\cdot\sigma\neq -\epsilon|y|$, we have
\begin{align}
|R_\sigma^\epsilon(y)\cdot\sigma|&=\epsilon\,(R_\sigma^\epsilon(y)\cdot\sigma),\label{sign}\\
|y|&=\frac{|R_{\sigma}^{\epsilon}(y)|}{|\widehat{R}_{\sigma}^{\epsilon}(y)\cdot\sigma|}\label{magnitude},\\
\widehat{y} \cdot \sigma &= \epsilon \left(2 |\widehat{R}_{\sigma}^{\epsilon}(y)\cdot \sigma|^2 -1\right). \label{angle}
\end{align}
Finally, for $y\in \R^3$ with $y\cdot\sigma\neq\pm |y|$, we have
\begin{equation}\label{R+ R- relation}
 |\widehat{R}_\sigma^+(y)|^2+|\widehat{R}_\sigma^-(y)|^2=1. 
\end{equation}
\end{proposition}

\subsection{Geometric identities}
The most pathological interactions occur when the frequencies of the interacting waves are on different scales and cannot pointwise offset the resonant cross-section. The next lemma provides an elementary, yet very useful lower bound for relating a small frequency with  a larger one, up to a singularity on the scattering direction $\sigma.$

\begin{lemma}
Let $k,k_1\in \R^3$ and $\sigma\in\S^2$. Then, the following point-wise estimates hold
\begin{align} \label{k^* lower bound by E}
\l k^* \r, \l k_1^* \r &>\left(\frac{1+E}{2}\right)^{1/2} \big(1 - \lambda_E\vert \widehat{K} \cdot \sigma \vert \big)^{1/2},\quad \lambda_E=\frac{E}{1+E},\\
\label{lower bound on v* by v1*}\l k^*\r&\geq \l k_1^*\r(1-\lambda_{k_1^*}|\widehat{k_1^*}\cdot\sigma|)^{1/2},\quad \lambda_{k_1^*}=\frac{|k_1^*|^2}{\l k_1^*\r^2},\\
\label{lower bound on v1* by v*}\l k_1^*\r&\geq \l k^*\r(1-\lambda_{k^*}|\widehat{k^*}\cdot\sigma|)^{1/2},\quad \lambda_{k^*}=\frac{|k^*|^2}{\l k^*\r^2},\\
\l k_1\r &> \frac{\l k\r}{3}\left(1-\lambda_{k}|\widehat{(k-2R_\sigma^\epsilon(w)}\cdot\sigma|\right)^{1/2},\quad \lambda_{k}=\frac{|k|^2}{\l k\r^2},\quad\epsilon\in\{+,-\},\label{lower bound on v1}\\
{\l k\r} &> {\frac{\l k_1\r}{3}\left(1-\lambda_{k_1}|\widehat{(k_1  + 2R_\sigma^\epsilon(w))}\cdot\sigma|\right)^{1/2},\quad \lambda_{k_1}=\frac{|k_1|^2}{\l k_1\r^2}, \quad \epsilon\in\{+,-\}}\label{lower bound on v}.
\end{align}
\end{lemma}
\begin{proof}
To prove \eqref{k^* lower bound by E}, we rely on the
elementary inequality
\begin{align}\label{energy inequality}
|w||K|&\leq \frac{|w|^2}{4}+|K|^2=\frac{E}{2}.
\end{align}
Then, by \eqref{collisional law} we have
\begin{align*}
|k^*|^2&=\frac{E}{2}-|w||K|(\widehat{K}\cdot\sigma) \geq \frac{E}{2}-|w||K||\widehat{K}\cdot\sigma| \geq \frac{E}{2}(1-|\widehat{K}\cdot\sigma|),
\end{align*}
which in turn implies
$$\l k^*\r^2=1+|k^*|^2\geq 1+\frac{E}{2}(1-|\widehat{K}\cdot\sigma|)> \frac{1+E}{2}\left(1-\lambda_E|\widehat{K}\cdot\sigma)|\right).$$
The bound for $k_1^*$ follows identically using the expression $|k_1^*|^2=\frac{E}{2}+|w||K|(\widehat{K}\cdot\sigma)$ instead. 

To prove \eqref{lower bound on v* by v1*}, we use \eqref{collisional law} to write $k_1^*-k^*=|w|\sigma$. Then we have
$$|k^*|^2=\Big|k_1^*-|w|\sigma\Big|^2\geq (|k_1^*|^2+|w|^2)(1-|\widehat{k_1^*}\cdot\sigma|),$$
which yields
$$\l k^*\r^2=1+|k^*|^2\geq 1+|k_1^*|^2(1-\widehat{k_1^*}\cdot\sigma)=\l k_1^*\r^2\left(1-\lambda_{k_1^*}|\widehat{k_1^*}\cdot\sigma|\right).$$
Bound \eqref{lower bound on v1* by v*} follows identically using the expression $k_1^*=k^*+|w|\sigma$ instead. 

To prove \eqref{lower bound on v1}, let us denote $\nu=R_\sigma^\epsilon(w)$. We use \eqref{inverse function} to write
$$k_1=k-w=\left(k-2\nu\right)+\frac{|\nu|}{(\widehat{\nu}\cdot\sigma)}\sigma,$$
which yields
\begin{align*}
|k_1|^2&\geq \left(|k-2\nu|^2+\frac{|\nu|^2}{|\widehat{\nu}\cdot\sigma|^2}\right)\left(1-|\widehat{(k-2\nu)}\cdot\sigma|\right)\notag\\
&\geq \left(|k-2\nu|^2+|\nu|^2\right)\left(1-|\widehat{(k-2\nu)}\cdot\sigma|\right).
\end{align*}
Now, we notice that $|k-2\nu|^2+|\nu|^2\geq \frac{|k|^2}{9}$. 
Indeed,
if $|\nu|\geq |k|/3$ the claim is immediate, while if $|\nu|<|k|/3$, the triangle inequality implies $|k-2\nu|\geq |k|-2|\nu|> |k|/3$. Combining these facts, we obtain
$$\l k_1\r^2=1+|k_1|^2\geq 1+\frac{|k|^2}{9}\left(1-|\widehat{(k-2\nu)\cdot\sigma|}\right)>\frac{\l k\r^2}{9}\left(1-\lambda_k|\widehat{(k-2\nu)}\cdot\sigma|\right),\quad\lambda_k=\frac{|k|^2}{\l k\r^2}.$$
Bound \eqref{lower bound on v} follows identically using the expression $k=w+k_1$ instead. The proof is complete.

\end{proof}

\subsection{Collisional averaging estimates}
We now present the averaging estimates that are necessary to offset the linear growth of the resonant cross-section. They are largely inspired by our previous work on the wave kinetic and Boltzmann equations \cite{AmLe24, AmLemain25}.

\begin{lemma}\label{collisional averging lemma}
Let $l>2$. Then for any $k,k_1\in\R^3$, the following averaging estimates hold
\begin{align}\label{collisional averaging individual}
\int_{\S^2}\frac{1}{\l k^*\r^l}\,d\sigma=\int_{\S^2}\frac{1}{\l k_1^*\r^l}\,d\sigma&\lesssim \frac{1}{(l-2)(1+E)}, \\
\int_{\S^2}\frac{1}{\l k^*\r^l\l k_1^*\r^l}\,d\sigma&\lesssim \frac{1}{(l-2)(1+E)^{1+l/2}}.\label{collisional avaraging coupled}
\end{align}
\end{lemma}
\begin{proof}
Fix $k,k_1\in\R^3$. Note that both statements are trivial if $E \leq 1.$ Thus we may assume that $E>1$, which implies that $\lambda_E=\frac{E}{1+E}\in(1/2,1)$. 

We first prove \eqref{collisional averaging individual}. By symmetry (changing $\sigma\mapsto -\sigma$), it is evident that $\int_{\S^2}\frac{1}{\l k^*\r^l}\,d\sigma=\int_{\S^2}\frac{1}{\l k_1^*\r^l}\,d\sigma$, so it suffices to show the estimate for $k^*$. Using \eqref{k^* lower bound by E} and integrating in spherical coordinates, we obtain
\begin{align*}
\int_{\S^2}\frac{1}{\l k^*\r^l}\,d\sigma&\lesssim \frac{1}{(1+E)^{l/2}}\int_{\S^2}\frac{1}{(1-\lambda_E|\widehat{K}\cdot\sigma|)^{l/2}}\,d\sigma  \approx  \frac{1}{(1+E)^{l/2}}\int_{0}^{2 \pi}\frac{1}{(1-\lambda_E|\cos \theta|)^{l/2}} \sin \theta \,d\theta \color{black} \\ 
& \approx \frac{1}{(1+E)^{l/2}}\int_0^1\frac{1}{(1-\lambda_E x)^{l/2}}\,d  x \color{black} =\frac{1}{\lambda_E(1+E)^{1/2}}\int_0^{\lambda_E}\frac{1}{(1-y)^{l/2}}\,dy\approx \frac{(1-\lambda_E)^{1-l/2}}{(l-2)\lambda_E(1+E)^{l/2}} \\
& \approx\frac{1}{(l-2)(1+E)},
\end{align*}
where we used the facts $\lambda_E\in(1/2,1)$ and $1-\lambda_E=\frac{1}{1+E}$. Estimate \eqref{collisional averaging individual} is proved.

Now, to prove \eqref{collisional avaraging coupled}, the resonant condition $|k^*|^2+|k_1^*|^2=E$ implies that either $|k^*|^2>E/2$ or $|k_1^*|^2>E/2$. Thus, using \eqref{collisional averaging individual}, we obtain
\begin{align*}
\int_{\S^2}\frac{1}{\l k^*\r^l\l k_1^*\r^l}\,d\sigma&=\int_{\S^2}\frac{1}{\l k^*\r^l\l k_1^*\r^l}\mathds{1}_{|k^*|^2>E/2}\,d\sigma+\int_{\S^2}\frac{1}{\l k^*\r^l\l k_1^*\r^l}\mathds{1}_{|k_1^*|^2>E/2}\,d\sigma \\
&\lesssim \frac{1}{(1+E)^{l/2}}\left(\int_{\S^2}\frac{1}{\l k_1^*\r^l}\,d\sigma+\int_{\S^2}\frac{1}{\l k^*\r^l}\,d\sigma\right)\lesssim\frac{1}{(l-2)(1+E)^{1+l/2}}.
\end{align*}

\end{proof}

\begin{lemma}
Let $1\leq p<\infty$. Then, the following estimates hold
\begin{align}
\sup_{k\in\R^3}\int_{\R^3\times\S^2}\frac{|h(k^*)|^p\chi(\widehat{w}\cdot\sigma)}{|\widehat{R}_\sigma^-(w)\cdot\sigma|^{2\alpha}}\,d\sigma\,dk_1&\lesssim \|h\|_{L^p}^p,\quad \alpha<1\label{precollisional on k^*}, \\
\sup_{k\in\R^3}\int_{\R^3\times\S^2}|\widehat{R}_\sigma^-(w)\cdot\sigma|^{2\alpha}|h(k_1^*)|^p\,d\sigma\,dk_1&\lesssim \|h\|_{L^p}^p,\quad \alpha>1/2\label{precollisional on k_1^*}.\\
\sup_{k\in\R^3}\int_{\R^3\times\S^2}\frac{|h(k_1)|^p}{(1-|\widehat{K}\cdot\sigma|)^\alpha}\,d\sigma\,dk_1&\lesssim \|h\|_{L^p}^p,\quad \alpha<1\label{estimate on integral of v1}.
\end{align}
\end{lemma}
\begin{proof} For the first two bounds, we rely on the identities
$$k^*=k-R_\sigma^+(w),\quad k_1^*=k-R_{\sigma}^-(w),\quad \widehat{w}\cdot\sigma=2|\widehat{R}^+_\sigma(w)\cdot\sigma|^2-1,\quad |\widehat{R}_\sigma^+(w)|^2+|\widehat{R}_\sigma^-(w)|^2=1,$$
straightforwardly deduced from \eqref{R+ R- v}, \eqref{angle} and \eqref{R+ R- relation}.

We first prove \eqref{precollisional on k^*}. Fix $k\in\R^3$. Using the substitution $y:=w=k-k_1$, we write
\begin{align*}
\int_{\R^3\times\S^2}\frac{|h(k^*)|^p\chi(\widehat{w}\cdot\sigma)}{|\widehat{R}_\sigma^-(w)\cdot\sigma|^{2\alpha}}\,d\sigma\,dk_1&=\int_{\R^3\times\S^2}\frac{|h(k-R_\sigma^+(w))|^p}{(1-|\widehat{R}_\sigma^+(w)|^2)^\alpha}\chi(2|\widehat{R}_\sigma^+(w)\cdot\sigma|^2-1)\,d\sigma\,dk_1\\
&=\int_{\R^3\times\S^2}\frac{|h(k-R_\sigma^+(y))|^p}{(1-|\widehat{R}_\sigma^+(y)\cdot\sigma|^2)^\alpha}\chi(2|\widehat{R}_\sigma^+(y)\cdot\sigma|^2-1)\,d\sigma\,dy.
\end{align*}
Now, using Proposition \ref{chgvarkin} to substitute $\nu:=R_\sigma^+(y)$, we obtain
\begin{align*}
\int_{\R^3\times\S^2}|h(k^*)|^p\chi(\widehat{w}\cdot\sigma)\,d\sigma\,dk_1&=\int_{\R^3} |h(k-\nu)|^p\left(\int_{\S^2}\frac{\chi(2|\widehat{\nu}\cdot\sigma|^2-1)}{|\widehat{\nu}\cdot\sigma|^2(1-|\widehat{\nu}\cdot\sigma|^2)^\alpha}\,d\sigma\right)\,d\nu\\
&\approx\int_{\R^3}|h(k-\nu)|^p\left(\int_0^1 \frac{\chi(2x^2-1)}{x^2(1-x^2)^{\alpha}}\,dx\right)\,d\nu\\
&\lesssim \int_{\R^3}|h(k-\nu)|^p\left(\int_0^1 \frac{1}{(1-x^2)^{\alpha}}\,dx\right)\,d\nu\\
& \leq \int_{\R^3}|h(k-\nu)|^p\left(\int_0^1 \frac{1}{(1-x)^{\alpha}}\,dx\right)\,d\nu\\
&\approx \int_{\R^3}|h(k-\nu)|^p\,d\nu=\|h\|_{L^p}^p,
\end{align*}
where we used the basic fact $\frac{\chi(2x^2-1)}{x^2}\lesssim 1$, $x\neq 0$ and the fact that $\alpha<1$ for the convergence of the integral in $x$.
Since $k$ is arbitrary, estimate \eqref{precollisional on k^*} follows.

We now prove \eqref{precollisional on k_1^*}. Fix $k\in\R^3$. Using the substitution $y:=w=k-k_1$, we obtain 
\begin{align*}
 \int_{\R^3\times\S^2}|\widehat{R}_\sigma^-(w)\cdot\sigma|^{2\alpha}|h(k_1^*)|^p\,d\sigma\,dk_1&=\int_{\R^3\times\S^2}|\widehat{R}_\sigma^-(w)\cdot\sigma|^{2\alpha}|h(k-R_\sigma(w))|^p\,d\sigma\,dk_1\\
 &=\int_{\R^3\times\S^2}|\widehat{R}_\sigma^-(y)\cdot\sigma|^{2\alpha}|h(k-R_\sigma(y))|^p\,d\sigma\,dy.
\end{align*}
Now, using Proposition \ref{chgvarkin} to substitute $\nu:=R_\sigma^-(y)$, we obtain
\begin{align*}
 \int_{\R^3\times\S^2}|\widehat{R}_\sigma^-(w)\cdot\sigma|^{2\alpha}|h(k_1^*)|^p\,d\sigma\,dk_1&=\int_{\R^3}|h(k-\nu)|^p\left(\int_{\S^2}|\widehat{\nu}\cdot\sigma|^{2(\alpha-1)}\,d\sigma\right)\,d\nu \\
 &\approx \int_{\R^3}|h(k-\nu)|^p\left(\int_0^1 x^{2(\alpha-1)}\,dx\right)\,d\nu\\
 &\approx \int_{\R^3}|h(k-\nu)|^p\,d\nu=\|h\|_{L^p}^p,
\end{align*}
where we used the fact $\alpha>1/2$ for the convergence of the integral in $x$.
Since $k$ is arbitrary, estimate \eqref{precollisional on k_1^*} follows.

Finally, to prove \eqref{estimate on integral of v1}, we fix $k\in\R^3$. Then we integrate in spherical coordinates using the identity
\begin{align}\label{integr-sph-identity}
    \int_{\S^2} F(k \cdot \sigma) \, d\sigma \approx \int_{0}^\pi F(\cos \theta) \, \sin \theta \, d\theta \approx \int_{-1}^1 F(x) \, dx
\end{align}
to obtain
\begin{align*}
    \int_{\R^3\times\S^2} \frac{|h(k_1)|^p}{(1-|\widehat{K}\cdot\sigma|)^\alpha}\,d\sigma\,k_1&= \int_{\R^3}|h(k_1)|^p\int_{\S^2}\frac{1}{(1-|\widehat{K}\cdot\sigma|)^\alpha}\,d\sigma\,dk_1\\
    &\approx\int_{\R^3}|h(k_1)|^p\int_{0}^1\frac{1}{(1-x)^\alpha}\,dx\,dk_1\\
    &\approx \|h \|_{L^p}^p,
\end{align*}
where we used the fact that $\alpha<1$ for the convergence of the integral in $x$. Since $k$ is arbitrary, estimate \eqref{estimate on integral of v1} follows.
\end{proof}

\subsection{An involutionary change of variables} Finally, we record a very useful change of variables property, which is the analog of the involutionary pre-post collisional velocities change of variables in the context of the Boltzmann equation. For a proof of this classical result, see e.g. Lemma 2.6 in \cite{AmLemain25}.

\begin{lemma} \label{involution}
 The map $T: (k,k_1,\sigma)\in\R^3\times\R^3\times\S^2\to (k^*,k_1^*,\eta)\in\R^3\times\R^3\times\S^2$ where
 \begin{equation}
 \begin{cases}
 k^*=K-\frac{|w|}{2}\sigma\\
 k_1^*=K+\frac{|w|}{2}\sigma\\
 \eta=-\widehat{w}
 \end{cases},\quad K=\frac{k+k_1}{2},\quad w=k-k_1,
\end{equation}
 is an involution of $\R^3\times\R^3\times\S^2$. Moreover, for any non-negative and continuously differentiable function $F:\R^3\times\R^3\times [-1,1]\to \R_+$, the following change of variables formula holds:
 \begin{equation}\label{change of variables formula}
\int_{\R^6\times\S^2}F(k^*,k_1^*,\widehat{w}\cdot\sigma)\,d\sigma\,dk_1\,dk=\int_{\R^6\times\S^2} F(k,k_1,\widehat{w}\cdot\sigma)\,d\sigma\,dk_1\,dk.  
 \end{equation}
 \end{lemma}

\section{Gain operators estimates } \label{sec:gain}
In this section we prove our main gain operator estimates. We stress that they are moment preserving, which will allow us to run a fixed point argument to prove the main result. At the technical level, the averaging in the angular variable quantified in Lemma \ref{collisional averging lemma} offsets the growth of the resonant cross-section.

\subsection{Estimate for $\mathcal{G}_0$} We prove the following estimate for $\mathcal{G}_0$:

\begin{proposition}\label{Prop G0}
  Let $r\geq 2$ and $l= 2-\frac{3}{r}+\delta$, where $0<\delta<1.$ Then the following estimate holds
  \begin{equation}\label{estimate for G0}
\|\l k\r^{l} \mathcal{G}_0[f,g,h]\|_{L^r}\lesssim \|\l k\r^{l} f\|_{L^r} \|\l k\r^l g\|_{L^r}\|\l k\r^l h\|_{L^r} .
  \end{equation}
\end{proposition}
\begin{proof} We treat the cases $2\leq r<\infty$ and $r=\infty$ separately.

\vspace{.3cm}
{\bf Case $r<\infty$.}
By the resonant condition $|k^*|^2+|k_1^*|^2=E$, we have  either $|k^*|^2\geq E/2$ or $|k_1^*|^2\geq E/2$, where $E=|k|^2+|k_1|^2$. Ignoring the measure zero set where $|k^*|^2=|k_1^*|^2=E/2$, we obtain
\begin{equation}\label{factorization of G0}
\mathcal{G}_0[f,g,h]=\frac{1}{4}f\left(\mathcal{Q}_0[g,h]+\mathcal{Q}_1[g,h]\right),
\end{equation}
where
\begin{align*}
\mathcal{Q}_0[g,h]&=\int_{\R^3\times\S^2}|w| |g(k^*)|\,|h(k_1^*)|\chi(\widehat{w}\cdot\sigma)\mathds{1}_{|k^*|>E/2}\,d\sigma\,dk_1,\\
\mathcal{Q}_1[g,h]&=\int_{\R^3\times\S^2}|w| |g(k^*)|\,|h(k_1^*)|\chi(\widehat{w}\cdot\sigma)\mathds{1}_{|k_1^*|>E/2}\,d\sigma\,dk_1.
\end{align*}
We clearly have
\begin{equation}\label{reduction to Linf G0}
\|\l k\r^l \mathcal{G}_0[f,g,h]\|_{L^r}\lesssim \|f_l\|_{L^r}\left(\|\mathcal{Q}_0[g,h]\|_{L^\infty}+\|\mathcal{Q}_1[g,h]\|_{L^\infty}\right),    
\end{equation}
so it suffices to estimate $\|\mathcal{Q}_0[g,h]\|_{L^\infty},\,\|\mathcal{Q}_1[g,h]\|_{L^\infty}$.

Define $l_0=\frac{1-\delta}{2}$, $l_1=\frac{1+\delta}{2}$. Clearly $l_0+l_1=1$. Moreover, we note that for any function $\psi$, H\"older's inequality implies
\begin{equation}\label{embedding L2 to Lr}
\|\psi_{l_0}\|_{L^2}\leq\|\psi_{1/2}\|_{L^2}\leq\|\psi_{l_1}\|_{L^2}\lesssim \|\l k\r^{\frac{1+\delta}{2}+3(\frac{1}{2}-\frac{1}{r})+\frac{\delta}{2}}\psi\|_{L^r}=\|\l k\r^{2-\frac{3}{r}+\delta}\psi\|_{L^r}=\|\psi_{l}\|_{L^r}.    
\end{equation}

\subsubsection*{Estimate for $\mathcal{Q}_0$} Fix $k\in\R^3$.
We use \eqref{lower bound on v1* by v*} and the inequality $|w|\leq |k|+|k_1|\lesssim E^{1/2}$ to bound
\begin{align*}
\frac{|w|}{\l k^*\r^{l_1}\l k_1^*\r^{l_0}}\mathds{1}_{|k^*|^2>E/2}\lesssim \frac{E^{1/2}}{\l k^*\r(1-|\widehat{k^*}\cdot\sigma|)^{l_0/2}}\mathds{1}_{|k^*|^2>E/2}\lesssim \frac{1}{(1-|\widehat{k^*}\cdot\sigma|)^{l_0/2}}.    
\end{align*}
Let $\alpha=1/2+\delta/4$, and using the above bound followed by the Cauchy-Schwarz inequality, we obtain
\begin{align*}
|\mathcal{Q}_0[g,h](k)|&\leq \int_{\R^3\times\S^2}\frac{|g_{l_1}(k^*)||h_{l_0}(k_1^*)|\chi(\widehat{w}\cdot\sigma)}{(1-|\widehat{k^*}\cdot\sigma|)^{l_0/2}}\,d\sigma\,dk_1\\
&=\int_{\R^3\times\S^2}\frac{|g_{l_1}(k^*)|\chi(\widehat{w}\cdot\sigma)}{|\widehat{R}_\sigma^-(w)\cdot\sigma|^{\alpha}(1-|\widehat{k^*}\cdot\sigma|)^{l_0/2}}\,\left(|\widehat{R}_\sigma^-(w)\cdot\sigma|^{\alpha}|h_{l_0}(k_1^*)|\right)\,d\sigma\,dk_1\\
&\leq \left(\int_{\R^3\times\S^2}\frac{|g_{l_1}(k^*)|^2\chi(\widehat{w}\cdot\sigma)}{|\widehat{R}_\sigma^-(w)\cdot\sigma|^{2\alpha}(1-|\widehat{k^*}\cdot\sigma|)^{l_0}}\,d\sigma\,dk_1\right)^{1/2}\left(\int_{\R^3\times\S^2}|\widehat{R}_\sigma^-(w)\cdot\sigma|^{2\alpha}|h_{l_0}(k_1^*)|^2\,d\sigma\,dk_1\right)^{1/2}\\
&\lesssim A^{1/2}\|h_{l_0}\|_{L^2},
\end{align*}
where 
\begin{align*}
    A & := \int_{\R^3\times\S^2}\frac{|g_{l_1}(k^*)|^2\chi(\widehat{w}\cdot\sigma)}{|\widehat{R}_\sigma^-(w)\cdot\sigma|^{2\alpha}(1-|\widehat{k^*}\cdot\sigma|)^{l_0}}\,d\sigma\,dk_1.
\end{align*}
For the last inequality, we used \eqref{precollisional on k_1^*} since $\alpha>1/2$.

It remains to estimate the integral $A$. For this, we  use \eqref{R+ R- v}, \eqref{angle} \eqref{R+ R- relation} and the substitution $y:=w=k-k_1$, to write
\begin{align*}
A&=\int_{\R^3\times\S^2}\frac{|g_{l_1}(k-R_\sigma^+(w))|^2\chi(2|\widehat{R}_\sigma^+(w)\cdot\sigma|^2-1)}{(1-|\widehat{R}_\sigma^+(w)\cdot\sigma|^2)^\alpha\left(1-|(\widehat{k-R_\sigma^+(w)})\cdot\sigma|\right)^{l_0}}\,d\sigma\,dk_1\\
&=\int_{\R^3\times\S^2}\frac{|g_{l_1}(k-R_\sigma^+(y))|^2\chi(2|\widehat{R}_\sigma^+(y)\cdot\sigma|^2-1)}{(1-|\widehat{R}_\sigma^+(y)\cdot\sigma|^2)^\alpha\left(1-|(\widehat{k-R_\sigma^+(y)})\cdot\sigma|\right)^{l_0}}\,d\sigma\,dy. 
\end{align*}
Now, using Proposition \ref{chgvarkin} to substitute $\nu:=R_\sigma^+(y)$, we obtain
\begin{align*}
A&= \int_{\R^3} |g_{l_1}(k-\nu)|^2\int_{\S^2}\frac{\chi(2|\widehat{\nu}\cdot\sigma)|^2-1)}{|\widehat{\nu}\cdot\sigma|^2(1-|\widehat{\nu}\cdot\sigma|^2)^\alpha\left(1-|\widehat{(k-\nu)}\cdot\sigma|\right)^{l_0}}\,d\sigma\,d\nu\\
&\lesssim \int_{\R^3} |g_{l_1}(k-\nu)|^2\int_{\S^2}\frac{1}{(1-|\widehat{\nu}\cdot\sigma|^2)^\alpha\left(1-|\widehat{(k-\nu)}\cdot\sigma|\right)^{l_0}}\,d\sigma\,d\nu\\
&\leq \int_{\R^3}|g_{l_1}(k-\nu)|^2\int_{\S^2}\frac{1}{(1-|\widehat{\nu}\cdot\sigma|^2)^{\alpha+l_0}}+\frac{1}{\left(1-|\widehat{(k-\nu)}\cdot\sigma|\right)^{\alpha+l_0}}\,d\sigma\,d\nu,
\end{align*}
where for the first bound we used the fact $\frac{\chi(2x^2-1)}{x^2}\lesssim 1$, $x\neq 0$
and for the second bound we used the standard convexity inequality 
\begin{equation}\label{standart singularity ineq}
\frac{1}{|a|^\lambda|b|^\mu}\leq\frac{1}{|a|^{\lambda+\mu}}+\frac{1}{|b|^{\lambda+\mu}},\quad a,b\neq 0,\quad \lambda,\mu>0.    
\end{equation}

Now integrating in spherical coordinates and using the fact that $\alpha+l_0=1-\delta/4$, we obtain
\begin{align*}
&\int_{\S^2}\frac{1}{(1-|\widehat{\nu}\cdot\sigma|^2)^{\alpha+l_0}}\,d\sigma\approx \int_0^1 \frac{1}{(1-x^2)^{1-\delta/4}}\,dx\leq \int_0^1\frac{1}{(1-x)^{1-\delta/4}}\,dx\approx 1, \\
&\int_{\S^2} \frac{1}{\left(1-|\widehat{(k-\nu)}\cdot\sigma|\right)^{\alpha+l_0}}\,d\sigma\approx \int_0^1\frac{1}{(1-x)^{1-\delta/4}}\,dx\approx 1.
\end{align*}
Hence 
$$A\lesssim \int_{\R^3}|g_{l_1}(k-\nu)|^2\,d\nu=\|g_{l_1}\|_{L^2}^2.$$
Since $k$ is arbitrary, we conclude
\begin{equation}\label{Q0 Linf bound}
\|\mathcal{Q}_0[g,h]\|_{L^\infty}\lesssim \|g_{l_1}\|_{L^2}\|h_{l_0}\|_{L^2}\lesssim \|g_l\|_{L^r}\|h_l\|_{L^r},    
\end{equation}
where for the last inequality we use \eqref{embedding L2 to Lr}.

\subsubsection*{Estimate for $\mathcal{Q}_1$} Fix $k\in\R^3$.
We use the inequality $\l k^*\r\geq \l k_1^*\r(1-|\widehat{k_1^*}\cdot\sigma|)$ (which follows by \eqref{lower bound on v* by v1*}), the fact $l_0+l_1=1$, and the inequality $|w|\leq |k|+|k_1|\lesssim E^{1/2}$, to bound
\begin{align*}
\frac{|w|}{\l k^*\r^{l_0}\l k_1^*\r^{l_1}}\mathds{1}_{|k_1^*|^2>E/2}\lesssim \frac{E^{1/2}}{\l k_1^*\r(1-|\widehat{k_1^*}\cdot\sigma|)^{l_0/2}}\mathds{1}_{|k_1^*|^2>E/2}\lesssim \frac{1}{(1-|\widehat{k_1^*}\cdot\sigma|)^{l_0/2}}.    
\end{align*}
Writing $\alpha=1-\delta/4$, and using the bound above followed by the Cauchy-Schwarz inequality, we obtain
\begin{align*}
|\mathcal{Q}_1[g,h](k)|&\leq \int_{\R^3\times\S^2}\frac{|g_{l_0}(k^*)||h_{l_1}(k_1^*)|\chi(\widehat{w}\cdot\sigma)}{(1-|\widehat{k_1^*}\cdot\sigma|)^{l_0/2}}\,\,d\sigma\,dk_1\\
&=\int_{\R^3\times\S^2}\frac{|g_{l_0}(k^*)|\chi(\widehat{w}\cdot\sigma)}{|\widehat{R}_\sigma^-(w)\cdot\sigma|^\alpha}\frac{|\widehat{R}_\sigma^-(w)\cdot\sigma|^\alpha|h_{l_1}(k_1^*)|}{(1-|\widehat{k_1^*}\cdot\sigma|)^{l_0/2}}\,\,d\sigma\,dk_1\\
&\leq \left(\int_{\R^3\times\S^2}\frac{|g_{l_0}(k^*)|^2\chi(\widehat{w}\cdot\sigma)}{|\widehat{R}_\sigma^-(w)\cdot\sigma|^{2\alpha}}\,d\sigma\,dk_1\right)^{1/2}\left(\int_{\R^3\times\S^2}\frac{|\widehat{R}_\sigma^-(w)\cdot\sigma|^{2\alpha}|h_{l_1}(k_1^*)|^2}{(1-|\widehat{k_1^*}\cdot\sigma|)^{l_0}}\,d\sigma\,dk_1\right)^{1/2}\\
&\lesssim \|g_{l_0}\|_{L^2}\,B^{1/2},
\end{align*}
where 
\begin{align*}
    B:= \int_{\R^3\times\S^2}\frac{|\widehat{R}_\sigma^-(w)\cdot\sigma|^{2\alpha}|h_{l_1}(k_1^*)|^2}{(1-|\widehat{k_1^*}\cdot\sigma|)^{l_0}}\,d\sigma\,dk_1.
\end{align*}
For the last inequality, we used \eqref{precollisional on k^*} since $\alpha<1$.

It remains to estimate the integral $B$. To do so, we first use \eqref{R+ R- v} and the substitution $y:=w=k-k_1$ to write
\begin{align*}
B&=\int_{\R^3\times\S^2}\frac{|\widehat{R}_\sigma^-(w)\cdot\sigma|^{2\alpha}|h_{l_1}(k-R_\sigma^-(w))|^2}{\left(1-|(\widehat{k-R_\sigma^-(w)})\cdot\sigma|\right)^{l_0}}\,d\sigma\,dk_1\\
&=\int_{\R^3\times\S^2}\frac{|\widehat{R}_\sigma^-(y)\cdot\sigma|^{2\alpha}|h_{l_1}(k-R_\sigma^-(y))|^2}{\left(1-|(\widehat{k-R_\sigma^-(y)})\cdot\sigma|\right)^{l_0}}\,d\sigma\,dy.
\end{align*}
Now, using Proposition \ref{chgvarkin} to substitute $\nu:=R_\sigma^-(y)$, we obtain
\begin{align*}
B&= \int_{\R^3} |h_{l_1}(k-\nu)|^2\int_{\S^2}\frac{1}{|\widehat{\nu}\cdot\sigma|^{2(1-\alpha)}\left(1-|\widehat{(k-\nu)}\cdot\sigma|\right)^{l_0}}\,d\sigma\,d\nu\\
&\leq \int_{\R^3}|h_{l_1}(k-\nu)|^2\int_{\S^2}\frac{1}{|\widehat{\nu}\cdot\sigma|^{2(1-\alpha)+l_0}}+\frac{1}{\left(1-|\widehat{(k-\nu)}\cdot\sigma|\right)^{2(1-\alpha)+l_0}}\,d\sigma\,d\nu,
\end{align*}
where for the last step we use inequality \eqref{standart singularity ineq}.
Now integrating in spherical coordinates and using the fact that $2(1-\alpha)+l_0=1/2$, we obtain
\begin{align*}
&\int_{\S^2}\frac{1}{|\widehat{\nu}\cdot\sigma|^{2(1-\alpha)+l_0}}\,d\sigma\approx \int_0^1 \frac{1}{x^{1/2}}\,dx\approx 1, \\
&\int_{\S^2} \frac{1}{\left(1-|\widehat{(k-\nu)}\cdot\sigma|\right)^{2(1-\alpha)+l_0}}\,d\sigma\approx \int_0^1\frac{1}{(1-x)^{1/2}}\,dx\approx 1.
\end{align*}
Hence 
$$B\lesssim \int_{\R^3}|h_{l_1}(k-\nu)|^2\,d\nu=\|h_{l_1}\|_{L^2}^2.$$
Since $k$ is arbitrary, we conclude
\begin{equation}\label{Q1 Linf bound}
\|\mathcal{Q}_1[g,h]\|_{L^\infty}\lesssim \|g_{l_0}\|_{L^2}\|h_{l_1}\|_{L^2}\lesssim \|g_l\|_{L^r}\|h_l\|_{L^r},    
\end{equation}
where for the last inequality we use \eqref{embedding L2 to Lr}.

Combining \eqref{Q0 Linf bound}, \eqref{Q1 Linf bound} with \eqref{reduction to Linf G0}, estimate \eqref{estimate for G0} for $2\leq r<\infty$ follows.

\vspace{.3cm}
{\bf Case $r=\infty$.} Let $k\in\R^3$. Since $r=\infty$, we have $l=2+\delta$. Using \eqref{collisional avaraging coupled} and the inequality $|w|\leq |k|+|k_1|\lesssim E^{1/2}$, we obtain
\begin{align*}
|\l k\r^l\mathcal{G}_0[f,g,h](k)|&\leq \|f_l\|_{L^\infty}\|g_l\|_{L^\infty}\|h_l\|_{L^\infty}\int_{\R^3}|w|\int_{\S^2}\frac{1}{\l k^*\r^l\l k_1^*\r^l}\,d\sigma\,dk_1\\
&\lesssim \|f_l\|_{L^\infty}\|g_l\|_{L^\infty}\|h_l\|_{L^\infty}\int_{\R^3}\frac{|w|}{E^{1+l/2}}\,dk_1\\
&\lesssim \|f_l\|_{L^\infty}\|g_l\|_{L^\infty}\|h_l\|_{L^\infty}\int_{\R^3}\frac{1}{\l k_1\r^{3+\delta}}\,dk_1\\
&\lesssim \|f_l\|_{L^\infty}\|g_l\|_{L^\infty}\|h_l\|_{L^\infty}.
\end{align*}
Since $k$ is arbitrary estimate \eqref{estimate for G0} follows.
\end{proof}

\subsection{Estimate for $\mathcal{G}_1$}\label{Prop G1} We  prove the following estimate for $\mathcal{G}_1:$
\begin{proposition}
  Let $r\geq 2$ and $l= 2-\frac{3}{r}+\delta$, where $0<\delta<1/r$ if $r < \infty,$ and $\delta>0$ if $r = \infty.$ Then the following estimate holds
  \begin{equation}\label{estimate on G1}
\|\l k\r^l \mathcal{G}_1[f,g,h]\|_{L^r}\lesssim \|\l k\r^l f\|_{L^r} \|\l k\r^l g\|_{L^r}\|\l k\r^l h\|_{L^r}.
  \end{equation}
\end{proposition}

\begin{proof}
We treat again the cases $2\leq r<\infty$ and $r=\infty$ separately.

{\bf Case $ r<\infty$.}
Since $r<\infty$ and $\delta<1/r$, we have $l=2-\frac{3}{r}+\delta<2$. We also have
\begin{equation}\label{r'l<2}
r'l=r'\left(2-\frac{3}{r}+\delta\right)=\frac{r}{r-1} \left(2-\frac{3}{r}+\delta\right)=\frac{2r-3+\delta r}{r-1}=2-\frac{1-\delta r}{r-1}<2.    
\end{equation}

We use the inequality
\begin{align}\label{lower-bound-product}
\l k^*\r\l k_1^*\r \gtrsim (1+E)(1-|\widehat{K}\cdot\sigma|)^{1/2}.
\end{align}

Indeed, by conservation of energy $\vert k_1^* \vert^2 + \vert k^* \vert ^2 = E,$ thus $\vert k_1^* \vert^2 \gtrsim E$ or $\vert k^* \vert^2 \gtrsim E.$ As a result $\max \big \lbrace \langle k_1^* \rangle, \langle k^* \rangle \big \rbrace \gtrsim (1+E)^{1/2}.$ Moreover, by \eqref{k^* lower bound by E}, we have $\min \big \lbrace \langle k_1^* \rangle, \langle k^* \rangle \big \rbrace \gtrsim (1+E)^{1/2} \big(1 - \vert \widehat{K} \cdot \sigma \vert \big)^{1/2}$ and \eqref{lower-bound-product} follows. 

\color{black}
Applying \eqref{lower-bound-product} then yields
\begin{align*}
\frac{\l k\r^l |w|}{\l k^*\r^{l} \l k_1^*\r^{l}}\lesssim \frac{1}{(1+E)^{\frac{l-1}{2}}(1-|\widehat{K}\cdot\sigma|)^{l/2}} \lesssim    \frac{1}{\l k_1\r^{l-1}(1-|\widehat{K}\cdot\sigma|)^{l/2}}.
\end{align*}
Using the above bound, we obtain
\begin{align}
\l k\r^l |\mathcal{G}_{1}[f,g,h]|&\lesssim \int_{\R^3\times\S^2}\frac{1}{(1-|\widehat{K}\cdot\sigma|)^{l/2}}|f_{1-l}(k_1)g_{l}(k^*)h_l(k_1^*)|\chi(\widehat{w}\cdot\sigma)\,d\sigma\,dk_1\notag\\
&\leq \left(\int_{\R^3\times\S^2} \frac{1}{(1-|\widehat{K}\cdot\sigma|)^{r'l/2}}|f_{1-l}(k_1)|^{r'}\,d\sigma\,dk_1\right)^{1/r'}\notag\\
&\hspace{1cm}\times \left(\int_{\R^3\times\S^2}|g_l(k^*)h_l(k_1^*)|^{r}\,d\sigma\,dk_1\right)^{1/r}\notag\\
&\lesssim \|f_{1-l}\|_{L^{r'}} \left(\int_{\R^3\times\S^2}|g_l(k^*)h_l(k_1^*)|^{r}\,d\sigma\,dk_1\right)^{1/r},\notag\\
&\lesssim \| f_l\|_{L^r} \left(\int_{\R^3\times\S^2}|g_l(k^*)h_l(k_1^*)|^{r}\,d\sigma\,dk_1\right)^{1/r}\label{pointwise bound on G1},
\end{align}
where for the second to last line we used the fact $r'l<2$ and \eqref{estimate on integral of v1}, and for the last line we used
$$ \|f_{1-l}\|_{L^{r'}}=\|\l k\r^{1-l} f\|_{L^{r'}}\lesssim \|\l k\r^{1-l +3\left(1-\frac{2}{r}\right)+\delta}f\|_{L^r}=\|\l k\r^{2-\frac{3}{r}}f\|_{L^r}\lesssim \| f_l\|_{L^r},$$
which follows by H\"older's inequality.

Raising \eqref{pointwise bound on G1} to the $r$-th power, integrating, and using \eqref{change of variables formula}, we obtain
\begin{align*}
\|\l k\r^l \mathcal{G}_1[f,g,h]\|_{L^r}^r&\lesssim \|f_l\|_{L^r}^r\int_{\R^6\times\S^2}|g_l(k^*) h_l(k_1^*)|^r\,d\sigma\,dk_1\,dk\\
&=\|f_l\|_{L^r}^r\int_{\R^6}|g_l(k) h_l(k_1)|^r\,\,dk_1\,dk\\
&= \|f_l\|_{L^r}^r\|g_l\|_{L^r}^r\|h_l\|_{L^r}^r,
\end{align*}
and \eqref{estimate on G1} follows.

{\bf Case $r=\infty$.} Let $k\in\R^3$.  Using \eqref{collisional avaraging coupled} and the inequality $|w|\lesssim E^{1/2}$, we bound 

\begin{align*}
\l k\r^l |\mathcal{G}_1[f,g,h](k)|
&= \|f_l\|_{L^\infty} \|g_l\|_{L^\infty} \|h_l\|_{L^\infty}\int_{\R^3}\frac{\l k\r^l |w|}{\l k_1\r^{l}}\int_{\S^2}\frac{1}{\l k^*\r^l \l k_1^*\r^l}\,d\sigma\,dk_1\\
&\lesssim \|f_l\|_{L^\infty} \|g_l\|_{L^\infty} \|h_l\|_{L^\infty} \int_{\R^3}\frac{\l k\r^l |w|}{\l k_1\r^{l}(1+E)^{1+l/2}}\,dk_1\\
 &\lesssim \|f_l\|_{L^\infty} \|g_l\|_{L^\infty} \|h_l\|_{L^\infty} \int_{\R^3}\frac{1}{\l k_1\r^{3+\delta}}\,dk_1\\
&\lesssim \|f_l\|_{L^\infty} \|g_l\|_{L^\infty} \|h_l\|_{L^\infty}.
\end{align*}
Since $k$ is arbitrary, estimate \eqref{estimate on G1} follows.
\end{proof}

\section{Loss operators estimates} \label{sec:loss}
In this section, we prove moment preserving trilinear estimates for the loss operators. Contrary to the Boltzmann equation, such a moment gain is present for the loss. As a result, the local well-posedness result is fully perturbative, in contrast to the Boltzmann equation.

\subsection{Estimate for $\mathcal{L}_0$} We  prove the following estimate for $\mathcal{L}_0$:

\begin{proposition}\label{Prop L0}
Let $r\geq 2$ and $l=2-\frac{3}{r}+\delta$, where $0<\delta<1$. Then, the following estimate holds
\begin{align}
\|\l k\r^l \mathcal{L}_0[f,g,h]\|_{L^r}&\lesssim \|\l k\r^{l} f\|_{L^r}\|\l k\r^l g\|_{L^r}\|\l k\r^l h\|_{L^r}\label{estimate for L_0}.
\end{align}    
\end{proposition}
\begin{proof}
We treat the cases $r<\infty=$ and $r=\infty$ separately.

{\bf Case $r<\infty$.}       We clearly have $\|\l k\r^l \mathcal{L}_0[f,g,h]\|_{L^r}\leq \|f_l\|_{L^r}\|\mathcal{R}_0[g,h]\|_{L^\infty}$. So it suffices to show that
\begin{equation}\label{sufficient condition L0}
\|\mathcal{R}_0[g,h]\|_{L^\infty}\lesssim \|g_l\|_{L^r}\|h_l\|_{L^r}.    
\end{equation}
We decompose
\begin{equation}
\mathcal{R}_0[g,h]= \mathcal{R}_0^0[g,h] +\mathcal{R}_0^1[g,h],    
\end{equation}
where
\begin{align*}
 \mathcal{R}_0^0[g,h]&=\frac{1}{4}\int_{\R^3\times\S^2}|w||g(k_1) h(k^*)|\mathds{1}_{|k_1|<|k|}\chi(\widehat{w}\cdot\sigma)\,d\sigma\,dk_1,\\
\mathcal{R}_0^1[g,h]&=\frac{1}{4}\int_{\R^3\times\S^2}|w||g(k_1)h(k^*)|\mathds{1}_{|k_1|\geq |k|}\chi(\widehat{w}\cdot\sigma)\,d\sigma\,dk_1.
\end{align*}

\subsubsection*{Estimate for $\mathcal{R}_0^1$} Fix $k\in\R^3$. We note that in the corresponding domain of integration, the triangle inequality implies $|w|\leq |k|+|k_1|\leq 2|k_1|$. Then,   estimate \eqref{k^* lower bound by E} and the Cauchy-Schwarz inequality imply
\begin{align*}
\mathcal{R}_0^1[g,h](k)&\lesssim \int_{\R^3\times\S^2}|k_1||g(k_1)||h(k^*)|\chi(\widehat{w}\cdot\sigma)\,d\sigma\,dk_1\\
&\lesssim \int_{\R^3\times\S^2}\frac{|g_{1/2}(k_1)|}{(1-|\widehat{K}\cdot\sigma|)^{1/4}} |h_{1/2}(k_1^*)|\chi(\widehat{w}\cdot\sigma))\,d\sigma\,dk_1\\
&\leq I_1^{1/2}I_2^{1/2},
\end{align*}
where 
\begin{align*}
I_1&=\int_{\R^3\times\S^2}   \frac{|g_{1/2}(k_1)|^2}{(1-|\widehat{K}\cdot\sigma)^{1/2}}\,d\sigma\,dk_1,\\
I_2&=\int_{\R^3\times\S^2} |h_{l_0}(k_1^*)|^2\chi(\widehat{w}\cdot\sigma)\,d\sigma\,dk_1.
\end{align*}

Now, estimate \eqref{estimate on integral of v1} implies $I_1\lesssim \|g_{1/2}\|_{L^2}^2$, while estimate \eqref{precollisional on k^*} implies
$I_2\lesssim \|h_{1/2}\|_{L^2}^2$.

Since $k$ is arbitrary, we conclude
\begin{equation}\label{bound on R01}
\|\mathcal{R}_0^1[g,h]\|_{L^\infty}\lesssim \|g_{1/2}\|_{L^2}\|h_{1/2}\|_{L^2}\lesssim \|g_l\|_{L^r}\|h_l\|_{L^r},    
\end{equation}
where for the last inequality we use \eqref{embedding L2 to Lr}.

\subsubsection*{Estimate for $\mathcal{R}_0^0$} Fix $k\in\R^3$. Using the triangle inequality $|w|\leq |k|+|k_1|\leq 2|k|$, estimates 
\eqref{lower bound on v1}, \eqref{k^* lower bound by E} and the Cauchy-Schwarz inequality, we obtain
\begin{align*}
 |\mathcal{R}_0^0[g,h](k)|&\lesssim \int_{\R^3\times\S^2}|k||g(k_1)||h(k^*)|\chi(\widehat{w}\cdot\sigma)\,d\sigma\,dk_1\\
 &\lesssim \int_{\R^3\times\S^2}\frac{|g_{1/2}(k_1)|}{(1-|\widehat{K}\cdot\sigma|)^{1/4}}\frac{|h_{1/2}(k^*)|}{\left(1-|\widehat{(k-2R_\sigma^-(w))}\cdot\sigma|\right)^{1/4}}\chi(\widehat{w}\cdot\sigma)\,d\sigma\,dk_1\\
 &\leq   J_1^{1/2}J_2^{1/2},
\end{align*}
where 
\begin{align*}
J_1&=\int_{\R^3\times\S^2}\frac{|g_{1/2}(k_1)|^2}{(1-|\widehat{K}\cdot\sigma|)^{1/2}}\,d\sigma\,dk_1,\\
J_2&=\int_{\R^3\times\S^2}\frac{|h_{1/2}(k^*)|^2\chi(\widehat{w}\cdot\sigma))}{\left(1-|\widehat{(k-2R_\sigma^+(w))}\cdot\sigma|\right)^{1/2}}\,d\sigma\,dk_1.
\end{align*}
By \eqref{estimate on integral of v1}, we have $J_1\lesssim \|g_{1/2}\|_{L^2}^2$.
For $J_2$, we use \eqref{R+ R- v}, \eqref{angle} and the substitution $y:=w=k-k_1$ to write
$$J_2=\int_{\R^3\times\S^2}\frac{|h_{1/2}(k-R_\sigma^+(y))|^2\chi(2|\widehat{R}_\sigma^+(y)\cdot\sigma|^2-1)}{\left(1-|\widehat{(k-2R_\sigma^+(y))}\cdot\sigma|\right)^{1/2}}\,d\sigma\,dy.$$

Using Proposition \ref{chgvarkin} to substitute $\nu:=R_\sigma^+(y)$, and the basic inequality $\frac{\chi(2z^2-1)}{z^2}\lesssim 1$, $z\neq 0$, we obtain
\begin{align*}
J_2&=\int_{\R^3}|h_{1/2}(k-\nu)|^2\int_{\S^2}\frac{\chi(|2\widehat{\nu}\cdot\sigma|^2-1)}{|\widehat{\nu}\cdot\sigma|^2\left(1-|\widehat{(k-\nu)}\cdot\sigma|)^{1/2}\right)}\,d\sigma\,d\nu  \\
&\approx \int_{\R^3}|h_{1/2}(k-\nu)|^2\int_{\S^2}\frac{1}{\left(1-|\widehat{(k-\nu)}\cdot\sigma|\right)^{1/2}}\,d\sigma\,d\nu \\
&\approx \int_{\R^3}|h_{1/2}(k-\nu)|^2\int_0^1\frac{1}{(1-x)^{1/2}}\,dx\,d\nu\\
&\approx \|h_{1/2}\|_{L^2}^2.
\end{align*}

Since $k$ is arbitrary, we conclude
\begin{equation}\label{bound on R00}
\|\mathcal{R}_0^0[g,h]\|_{L^\infty}\lesssim \|g_{1/2}\|_{L^2}\|h_{1/2}\|_{L^2}\lesssim \|g_l\|_{L^r}\|h_l\|_{L^r},    
\end{equation}
where for the last inequality we use \eqref{embedding L2 to Lr}.

Combining \eqref{bound on R01}, \eqref{bound on R00}, we obtain \eqref{sufficient condition L0} and \eqref{estimate for L_0} follows.

{\bf Case $r=\infty$.} Fix $k\in\R^3$.   Using  \eqref{collisional averaging individual} and the inequality $|w|\lesssim E^{1/2}$, we bound
\begin{align*}
|\l k\r^l f(k)\,\mathcal{R}_0[g,h](k)|&\leq \|f_l\|_{L^\infty}\|g_{l}\|_{L^\infty}\|h_{l}\|_{L^\infty}\int_{\R^3}\frac{|w|}{\l k_1\r^{l}}\int_{\S^2}\frac{1}{\l k^*\r^{l}}\,d\sigma\,dk_1\\
&\lesssim \|f_l\|_{L^\infty}\|g_{l}\|_{L^\infty}\|h_{l}\|_{L^\infty}\int_{\R^3}\frac{|w|}{\l k_1\r^{l}(1+E)}\,dk_1\\
&\lesssim \|f_l\|_{L^\infty}\|g_{l}\|_{L^\infty}\|h_{l}\|_{L^\infty}\int_{\R^3}\frac{1}{\l k_1\r^{3+\delta}}\,dk_1\\
&\lesssim \|f_l\|_{L^\infty}\|g_{l}\|_{L^\infty}\|h_{l}\|_{L^\infty}.
\end{align*}
Since $k$ is arbitrary, estimate \eqref{estimate for L_0} follows.

\end{proof}

\subsection{Estimate for $\mathcal{L}_1$} We  prove the following estimate for $\mathcal{L}_1$:
\begin{proposition}\label{Prop L1}
Let $r\geq 2$ and $l=2-\frac{3}{r}+\delta$, where $0<\delta<1$. Then, the following estimate holds
\begin{align}
\|\l k\r^l  \mathcal{L}_1[f,g,h]\|_{L^r}&\lesssim \|\l k\r^{l} f\|_{L^r}\|\l k\r^l g\|_{L^r}\|\l k\r^l h\|_{L^r}\label{estimate for L_1}.
\end{align}    
\end{proposition}
\begin{proof}
We treat the cases $r<\infty=$ and $r=\infty$ separately again.

{\bf Case $ r<\infty$.}   We clearly have $\|\l k\r^l\mathcal{L}_1[f,g,h]\|_{L^r}\leq \|f_l\|_{L^r}\|\mathcal{R}_1[g,h]\|_{L^\infty}$. So it suffices to show that
\begin{equation}\label{sufficient condition L1}
\|\mathcal{R}_1[g,h]\|_{L^\infty}\lesssim \|g_l\|_{L^r}\|h_l\|_{L^r}.    
\end{equation}
 We decompose
\begin{equation}
\mathcal{R}_1[g,h]= \mathcal{R}_1^0[g,h] +\mathcal{R}_1^1[g,h],    
\end{equation}
where
\begin{align*}
 \mathcal{R}_1^0[g,h]&=\frac{1}{4}\int_{\R^3\times\S^2}|w||g(k_1) h(k_1^*)|\mathds{1}_{|k_1|<|w|/2}\chi(\widehat{w}\cdot\sigma)\,d\sigma\,dk_1,\\
\mathcal{R}_1^1[g,h]&=\frac{1}{4}\int_{\R^3\times\S^2}|w||g(k_1)h(k_1^*)|\mathds{1}_{|k_1|\geq |w|/2}\chi(\widehat{w}\cdot\sigma)\,d\sigma\,dk_1.
\end{align*}

Let $l_0=\frac{1-\delta}{2}$, $l_1=\frac{1+\delta}{2}$.

\subsubsection*{Estimate for $\mathcal{R}_1^1$} Let $\alpha=1/2+\delta/4,$ and fix $k\in\R^3$. We use estimate \eqref{k^* lower bound by E} and the Cauchy-Schwarz inequality to write
\begin{align*}
\mathcal{R}_1^1[g,h](k)&\lesssim \int_{\R^3\times\S^2}|k_1||g(k_1)||h(k_1^*)|\chi(\widehat{w}\cdot\sigma)\,d\sigma\,dk_1\\
&= \int_{\R^3\times\S^2} |g_{l_1}(k_1)|\frac{|k_1|^{l_0}|h_{l_0}(k_1^*)|}{\l k_1^*\r^{l_0}}\chi(\widehat{w}\cdot\sigma)\,d\sigma\,dk_1\\
&\lesssim \int_{\R^3\times\S^2}\frac{|g_{l_1}(k_1)|\chi(\widehat{w}\cdot\sigma)}{|\widehat{R}_\sigma^-(w)\cdot\sigma|^\alpha(1-|\widehat{K}\cdot\sigma|)^{l_0/2}} \left(|\widehat{R}_\sigma^-(w)\cdot\sigma|^\alpha|h_{l_0}(k_1^*)|\right)\,d\sigma\,dk_1\\
&\leq I_1^{1/2}I_2^{1/2},
\end{align*}
where 
\begin{align*}
I_1&=\int_{\R^3\times\S^2}   \frac{|g_{l_1}(k_1)|^2\chi(\widehat{w}\cdot\sigma)}{|\widehat{R}_\sigma^-(w)\cdot\sigma|^{2\alpha}(1-|\widehat{K}\cdot\sigma|)^{l_0}}\,d\sigma\,dk_1,\\
I_2&=\int_{\R^3\times\S^2} |\widehat{R}_\sigma^-(w)\cdot\sigma|^{2\alpha}|h_{l_0}(k_1^*)|^2\,d\sigma\,dk_1.
\end{align*}

Since $\alpha>1/2$, estimate \eqref{precollisional on k_1^*} implies
$I_2\lesssim \|h_{l_0}\|_{L^2}^2$

To estimate $I_1$, we use \eqref{angle} followed by \eqref{standart singularity ineq} to write
\begin{align*}
I_1&=\int_{\R^3}|g_{l_1}(k_1)|^2\int_{\S^2}\frac{\chi(\widehat{w}\cdot\sigma)}{(1-\widehat{w}\cdot\sigma)^\alpha (1-|\widehat{K}\cdot\sigma|)^{l_0}}\,d\sigma\,dk_1\\ 
&\leq\int_{\R^3}|g_{l_1}(k_1)|^2\int_{\S^2}\frac{\chi(\widehat{w}\cdot\sigma)}{(1-\widehat{w}\cdot\sigma)^{\alpha+l_0}}+\frac{1}{(1-|\widehat{K}\cdot\sigma|)^{\alpha+l_0}}\,d\sigma\,dk_1.
\end{align*}
Integrating in spherical coordinates, we obtain
\begin{align*}
\int_{\S^2} \frac{\chi(\widehat{w}\cdot\sigma)}{(1-\widehat{w}\cdot\sigma)^{\alpha+l_0}}\,d\sigma&\approx\int_{0}^1\frac{1}{(1-x)^{\alpha+l_0}}\,dx\lesssim 1,\\   
\int_{\S^2}\frac{1}{(1-|\widehat{K}\cdot\sigma|)^{\alpha+l_0}}\,d\sigma &\approx \int_0^1 \frac{1}{(1-x)^{\alpha+l_0}}\,dx\lesssim 1,
\end{align*}
where we used the fact that $\alpha+l_0=1-\delta/4$ for the convergence of the integrals in $x$. It follows that $I_1\lesssim \|g_{l_1}\|_{L^2}^2$.

Since $k$ is arbitrary, we conclude
\begin{equation}\label{bound on R11}
\|\mathcal{R}_1^1[g,h]\|_{L^\infty}\lesssim \|g_{l_1}\|_{L^2}\|h_{l_0}\|_{L^2}\lesssim \|g_l\|_{L^r}\|h_l\|_{L^r},    
\end{equation}
where for the last inequality we use \eqref{embedding L2 to Lr}.

\subsubsection*{Estimate for $\mathcal{R}_1^0$} 
    Define $\alpha=1-\delta/4$ and fix $k\in\R^3$. Since in the corresponding domain of integration $\widehat{w}\cdot\sigma>0$ and $|k_1|<|w|/2$, \eqref{R+ R- v}, the triangle inequality, and \eqref{angle} imply 
\begin{align*}
|k_1^*|&=|k_1-R_\sigma^+(w)|\geq |R_\sigma^+(w)|-|k_1|\geq |w||\widehat{R}_\sigma^+(w)\cdot\sigma|-\frac{|w|}{2}\notag\\
&=|w|\sqrt{\frac{1+\widehat{w}\cdot\sigma}{2}}-\frac{|w|}{2}>\frac{\sqrt{2}-1}{2}|w|.
\end{align*}
In addition to that, by the triangle inequality we have $|k|=|u+k_1|\geq |w|-|k_1|\geq |w|/2$. Using \eqref{lower bound on v1} for $\epsilon=-$ as well, we obtain
$$\frac{|w|}{\l k_1\r^{l_0}\l k_1^*\r^{l_1}}\mathds{1}_{|k_1|<|w|/2}\chi(\widehat{w}\cdot\sigma)\lesssim \frac{1}{\left(1-|\widehat{(k-2R_\sigma^-(w))}\cdot\sigma|\right)^{l_0/2}}.$$
Now, the above estimate followed the Cauchy-Schwarz inequality imply
\begin{align*}
|\mathcal{R}_1^0[g,h](k)|&\lesssim \int_{\R^3\times\S^2}\frac{|g_{l_0}(k_1)||h_{l_1}(k_1^*)|\chi(\widehat{w}\cdot\sigma)}{\left(1-|\widehat{(k-2R_\sigma^-(w))}\cdot\sigma|\right)^{l_0/2}}\,d\sigma\,dk_1\\
&=\int_{\R^3\times\S^2}\frac{|g_{l_0}(k_1)|\chi(\widehat{w}\cdot\sigma)}{|\widehat{R}_\sigma^-(w)\cdot\sigma|^\alpha} \frac{|\widehat{R}_\sigma^-(w)\cdot\sigma|^\alpha|h_{l_1}(k^*)|}{\left(1-|\widehat{(k-2R_\sigma^-(w))}\cdot\sigma|\right)^{l_0/2}}\,d\sigma\,dk_1\\
&\leq J_1^{1/2}J_2^{1/2},
\end{align*}
where
\begin{align*}
J_1&=\int_{\R^3\times\S^2} \frac{|g_{l_0}(k_1)|^2\chi(\widehat{w}\cdot\sigma)}{|\widehat{R}_\sigma^-(w)\cdot\sigma|^{2\alpha}} \,d\sigma\,dk_1, \\
J_2&= \int_{\R^3\times\S^2} \frac{|\widehat{R}_\sigma^-(w)\cdot\sigma|^{2\alpha}|h_{l_1}(k^*)|^2}{\left(1-|\widehat{(k-2R_\sigma^-(w))}\cdot\sigma|\right)^{l_0}}\,d\sigma\,dk_1.
\end{align*}

For $J_1$, we use \eqref{angle} to write
\begin{equation*}
J_1\approx \int_{\R^3}|g_{l_0}(k_1)|^2\int_{\S^2}\frac{\chi(\widehat{w}\cdot\sigma)}{(1-\widehat{w}\cdot\sigma)^\alpha}\,d\sigma\,dk_1\approx \int_{\R^3}|g_{l_0}(k_1)|^2\int_{0}^1\frac{1}{(1-x)^\alpha}\,dx\,dk_1 \approx \|g_{l_0}\|_{L^2}^2, 
\end{equation*}
where we used the fact that $\alpha<1$ for the convergence of the integral in $x$.

For $J_2$, we first use \eqref{R+ R- v} and the substitution $y:=w=k-k_1$ to obtain
$$J_2=\int_{\R^3\times\S^2} \frac{|\widehat{R}_\sigma^-(y)\cdot\sigma|^{2\alpha}|h_{l_1}(k-R_\sigma^-(y))|^2}{\left(1-|\widehat{(k-2R_\sigma^-(y))}\cdot\sigma|\right)^{l_0}}\,d\sigma\,dy.$$
Then we use Proposition \ref{chgvarkin} to substitute $\nu:=R_\sigma^-(y)$ and obtain
\begin{align*}
J_2&\approx \int_{\R^3}|h_{l_1}(k-\nu)|^2\int_{\S^2}\frac{1}{|\widehat{\nu}\cdot\sigma|^{2(1-\alpha)}\left(1-|\widehat{(k-\nu)}\cdot\sigma|\right)^{l_0}}\,d\sigma\,d\nu\\
&\leq \int_{\R^3}|h_{l_1}(k-\nu)|^2\int_{\S^2}\frac{1}{|\widehat{\nu}\cdot\sigma|^{2(1-\alpha)+l_0}}+\frac{1}{\left(1-|\widehat{(k-\nu)}\cdot\sigma|\right)^{2(1-\alpha)+l_0}}\,d\sigma\,d\nu,
\end{align*}
where for the last step, we used \eqref{standart singularity ineq}.
Now, integrating in spherical coordinates, we obtain
\begin{align*}
 &\int_{\S^2}\frac{1}{|\widehat{\nu}\cdot\sigma|^{2(1-\alpha)+l}}\,d\sigma\approx\int_0^1\frac{1}{x^{2(1-\alpha)+l_0}}\,dx\lesssim 1,\\
 &\int_{\S^2}\frac{1}{\left(1-|\widehat{(k-\nu)}\cdot\sigma|\right)^{2(1-\alpha)+l_0}}\,d\sigma\approx \int_0^1\frac{1}{(1-x)^{2(1-\alpha)+l_0}}\,dx\lesssim 1,
\end{align*}
where we used the fact that $2(1-\alpha)+l_0=1/2$ for the convergence of the integral in $x$. It follows that $J_2\lesssim \|h_{l_1}\|_{L^2}^2$.

Since $k$ is arbitrary, we conclude
\begin{equation}\label{bound on R10}
\|\mathcal{R}_1^0[g,h]\|_{L^\infty}\lesssim \|g_{l_0}\|_{L^2}\|h_{l_1}\|_{L^2}\lesssim \|g_l\|_{L^r}\|h_l\|_{L^r},    
\end{equation}
where for the last inequality we use \eqref{embedding L2 to Lr}.

Combining \eqref{bound on R11}, \eqref{bound on R10}, we obtain \eqref{sufficient condition L1}, and \eqref{estimate for L_1} follows. 

{\bf Case $r=\infty$.} The proof is identical to the corresponding case  of Proposition \ref{Prop L0}.
\end{proof}

\section{Proof of the main result} \label{sec:pfmain}
In this section, we use the estimates from the previous Sections \ref{sec:gain} and \ref{sec:loss} to prove our main result Theorem \ref{main-thm}. We fix $2 \leq r \leq \infty$  and denote $l = 2 - \frac{3}{r} + \delta$, where $0 < \delta < 1/r$  if $r<\infty$, and $\delta>0$ if $r=\infty$.

We first note the following auxiliary estimate: for $\phi,\psi\in \l v\r^{-l}L^r$ and $\mathcal{T}\in\{\mathcal{G}_0,\mathcal{G}_1,\mathcal{L}_0,\mathcal{L}_1\}$, trilinearity, the triangle inequality and \eqref{estimate for G0}, \eqref{estimate on G1}, \eqref{estimate for L_0}, \eqref{estimate for L_1} imply that for some numerical constant $C$ we have

\begin{equation}\label{aux bound}
\begin{aligned}
\|&\l k\r^l\left(\mathcal{T}[\phi,\phi,\phi]-\mathcal{T}[\psi,\psi,\psi]\right)\|_{L^r}\\
&\leq \|\l k\r^l\mathcal{T}[\phi-\psi,\phi,\phi]\|_{L^r}+\|\l k\r^l\mathcal{T}[\psi,\phi-\psi,\phi]\|_{L^r} +  \|\l k\r^l\mathcal{T}[\psi,\psi,\phi-\psi]\|_{L^r}\\
&\leq C\left(\|\l k\r^l\phi\|_{L^r}^2+\|\l k\r^l\phi\|_{L^r}\|\l k\r^l\psi\|_{L^r}+\|\l k\r^l\psi\|_{L^r}^2\right)\|\l k\r^l(\phi-\psi)\|_{L^r}.
\end{aligned}
\end{equation}
In particular, for $\psi=0$ we have
\begin{equation}\label{aux-zero}
    \|\l k\r^l\mathcal{T}[\phi,\phi,\phi]\|_{L^r}\leq C\|\l k\r^l\phi\|_{L^r}^3.
\end{equation}

\subsection{Proof of local well-posedness} \label{subsec:LWP}

Fix $f_0 \in \l k \r^{-l} L^r$, and let $R := \Vert \l k \r^{l} f_0 \Vert_{L^{r}},$ $T:= \frac{1}{96 C R^2}$ where $C$ was defined in \eqref{aux bound}. Define
\begin{align*}
\textbf{B}(2R) := \bigg \lbrace f \in \mathcal{C} \big( [0,T]; \l k \r^{-l} L^{r} \big) \,:\, \sup_{t \in [0,T]} \Vert \l k \r^{l} f(t) \Vert_{L^{r}} \leq 2R \bigg \rbrace.
\end{align*}

Consider the map
\begin{align*} \Phi_{f_0}: 
    \begin{cases}
    \textbf{B}(2R) & \longrightarrow \, \mathcal{C}\big([0,T];\l k\r^{-l}L^r\big) \\
    f & \longmapsto \,  f_0 + \displaystyle \int_0^t \mathcal{C}[f](s) \, ds
    \end{cases} 
\end{align*}
We prove that $\Phi_{f_0}$ is a contraction on $\textbf{B}(2R).$ 
\\

\noindent
\textit{Stability.} Let $f \in \textbf{B}(2R).$ By the triangle inequality and \eqref{aux-zero}, we obtain
\begin{align*}
    \Vert \l k \r^{l} \Phi_{f_0} (f)(t) \Vert_{L^r} &\leq \|\l k\r^l f_0\|_{L^r}+\sum_{\mathcal{T}\in\{\mathcal{G}_0,\mathcal{G}_1,\mathcal{L}_0,\mathcal{L}_1\}}\int_0^t\|\l k\r^l\mathcal{T}[f,f,f](s)\|_{L^r}\,ds\\
    &\leq  
    R + 32 C T R^3 < 2R.
\end{align*}
Therefore $\Phi_{f_0}:\bm{B}(2R)\to\bm{B}(2R)$.
\\

\noindent
\textit{Contraction.} 
Let $f,g\in\bm{B}(2R)$. Then, by the triangle inequality and \eqref{aux bound}, we obtain

\begin{align*}
    \big \Vert \l k \r^{l} \big( \Phi_{f_0} (f) - \Phi_{f_0}(g) \big) \big \Vert_{L^r}&\leq \sum_{\mathcal{T}\in\{\mathcal{G}_0,\mathcal{G}_1,\mathcal{L}_0,\mathcal{L}_1\}}\int_0^t\big\|\l k\r^l\big(\mathcal{T}[f,f,f](s)-\mathcal{T}[g,g,g](s)\big)\big\|_{L^r}\,ds\\
&\leq 48 C R^2 T \big \Vert \l k \r^{l} \big( f-g \big) \big \Vert_{L^r} = \frac{1}{2} \big \Vert \l k \r^{l} \big( f-g \big) \big \Vert_{L^r}.
\end{align*}
Existence and uniqueness of a solution to \eqref{KWE} in $[0,T]$ follows by the contraction mapping principle.
\\

\noindent 
\textit{Continuous dependence on the initial data.} Let $f_0, g_0 \in\l k\r^{-l}L^r$. Denote $R_1:=\|\l k\r^l f_0\|_{L^r}$, $R_2:=\|\l k\r^{l}g_0\|_{L^r}$ and let $T_1=\frac{1}{96CR_1^2}$, $T_2=\frac{1}{96CR_2^2}$ be the corresponding times of existence. Let $f\in\bm{B}(2R_1)$ and $g\in\bm{B}(2R_2)$ be the solutions corresponding to  $f_0$ and $g_0$ respectively, and denote $T_{min}:=\min\{T_1,T_2\}=\frac{1}{96C\max\{R_1^2,R_2^2\}}$.

 Then, for all $t\in[0,T]$, the definition of a solution, the triangle inequality, and \eqref{aux bound} imply
\begin{align*} 
&\big \Vert \l k \r^{l} \big( f(t) - g(t) \big) \big \Vert_{L^r} \\
&\hspace{.5cm}\leq \big \Vert \l k \r^{l} \big( f_0 - g_0 \big)\|_{L^r}+ \sum_{\mathcal{T}\in\{\mathcal{G}_0,\mathcal{G}_1,\mathcal{L}_0,\mathcal{L}_1\}}\int_0^t\big\|\l k\r^l\big(\mathcal{T}[f,f,f](s)-\mathcal{T}[g,g,g](s)\big)\big\|_{L^r}\,ds\\
&\leq \big \Vert \l k \r^{l} \big( f_0 - g_0 \big)\|_{L^r}+48C\big(\max\{R_1,R_2\}\big)^2T_{min} \sup_{t\in[0,T_{min}]}\|\l k\r^l \big(f(t)-g(t)\big)\|_{L^r}\\
&=\big \Vert \l k \r^{l} \big( f_0 - g_0 \big)\|_{L^r}+\frac{1}{2} \sup_{t\in[0,T_{min}]}\|\l k\r^l \big(f(t)-g(t)\big)\|_{L^r}.
\end{align*}
Taking supremum on the left hand side of the above estimate, \eqref{continuity wrt data estimate} follows.

\subsection{Proof of positivity}

We now prove the last part of the statement in Theorem \ref{main-thm}, namely that the flow preserves positivity. Again denote $R=\|\l k\r^l f_0\|_{L^r}$ and $T=\frac{1}{96CR^2}$.

We rely on the classical idea of the Kaniel-Shinbrot iteration \cite{KS,illner-shinbrot}, which we used in our previous work \cite{AmLe24} to prove positivity as well. We sketch the approach for the convenience of the reader.
\\

\noindent
\textit{Definition of lower and upper solutions.} Let $u_0, l_0:[0,\infty)\times\R^3\to\R$ with $0\leq l_0\leq u_0$. Recalling \eqref{defQ+} and \eqref{defR}, define the coupled initial value problems
\begin{equation}\label{lower IVP n}
\begin{cases}
\partial_t l_n+ l_n \mathcal{R}[u_{n-1}] &=\mathcal{Q}^{+}[l_{n-1}]\\
l_n(0)=f_0
\end{cases},\quad n\in\mathbb{N},
\end{equation}
and
\begin{equation}\label{upper IVP n}
\begin{cases}
\partial_t u_n+ u_n \mathcal{R}[l_{n-1}]&=\mathcal{Q}^{+}[u_{n-1}]\\
u_n(0)=f_0
\end{cases},\quad  n\in \mathbb{N}.
\end{equation}
After integration, we find the integro-differential equations
\begin{equation}
\label{solution lower IVP n}
l_n(t)= \exp \left( -\int_0^t \mathcal{R}[u_{n-1}](s)\,ds \right) f_0 +\int_0^t \mathcal{Q}^{+}[l_{n-1}](s)\exp\left(-\int_s^t \mathcal{R}[u_{n-1}](\tau)\,d\tau \right)\,ds,
\end{equation}
\begin{equation}\label{solution upper IVP n}
u_n(t)=\exp\left( -\int_0^t \mathcal{R}[l_{n-1}](s)\,ds\right) f_0 +\int_0^t \mathcal{Q}^{+}[u_{n-1}](s)\exp\left(-\int_s^t \mathcal{R}[l_{n-1}](\tau)\,d\tau \right)\,ds.
\end{equation}

\noindent
\textit{Initialization.} We define $l_0 = 0$ and $u_0 = \widetilde{f}$, where $\widetilde{f}$ is the strong solution of the gain only initial value problem (recalling the notation \eqref{defQ+})
\begin{align} \label{IVP Q+}
    \begin{cases}
    \partial_t f = \mathcal{Q}^+[f], \\
    f(0) = f_0.
    \end{cases}
\end{align}
One can show that this problem is well-posed in $[0,T]$, and that the solution $\widetilde{f}$ satisfies 
\begin{equation}\label{gain solution estimate}
\sup_{t\in[0,T]}\|\l k\r^l \widetilde{f}(t)\|_{L^r}\leq 2R.    
\end{equation}
The proof is identical to the proof in Subsection \ref{subsec:LWP} (just ignore the loss terms), so we omit the details.
 Moreover $\widetilde{f} \geq 0$ by the positivity of successive Picard iterates. 
\\

\noindent 
\textit{Monotonicity.} 

It is standard to prove that the sequences $(l_n)_{n=1}^\infty,\,(u_n)_{n=1}^\infty$ are nested, that is for every $n \in \mathbb{N},$ 
\begin{equation}\label{nesting condition}
0= l_0\leq l_1\leq\dots\leq l_{n-1}\leq l_n\leq u_n\leq u_{n-1}\leq\dots \leq u_1= u_0.
\end{equation}
Indeed, this follows directly by induction, combined with the integral representations \eqref{solution lower IVP n}-\eqref{solution upper IVP n}, as well as the monotonicity of the operators $\mathcal{R}$ and $\mathcal{Q}^{+}.$
\color{black}
\\

\noindent 
\textit{Convergence.} For fixed $t\geq 0$, \eqref{nesting condition} implies that the sequence $(l_n(t))_{n=0}^\infty$ is increasing and upper bounded, therefore it converges $l_n(t)\nearrow l(t)$. Similarly, the sequence  $(u_n(t))_{n=0}^\infty$ is decreasing and lower bounded, so $u_n(t)\searrow u(t)$. Clearly, by \eqref{nesting condition} we have
\begin{equation}\label{comparison u, l}
0\leq l(t)\leq u(t)\leq u_0.
\end{equation}

Next we integrate \eqref{lower IVP n}-\eqref{upper IVP n} in time and let $n\to\infty.$ By the dominated convergence theorem, we obtain
\begin{align}
l(t)+ \int_0^t  \big(l\mathcal{R}[u]\big)(s) \,ds&=f_0+\int_0^t \mathcal{Q}^{+}[l](s) \,ds,\label{integrated loss}\\
u(t)+ \int_0^t \big( u\mathcal{R}[l] \big)(s)\,ds &=f_0+\int_0^t \mathcal{Q}^{+}[u](s)\,ds.\label{integrated gain}
\end{align}
Subtracting \eqref{integrated gain} from \eqref{integrated loss} and using the fact that $u\geq l$, we find that $w := u-l\geq 0$ satisfies
\begin{align}
 w(t)&=\int_0^t \left( \mathcal{Q}^{+}[u]-\mathcal{Q}^{+}[l] \right)(s)\,ds-\int_{0}^t (u\mathcal{R}[l]-l\mathcal{R}[u])(s)\,ds\notag\\
 &\leq \int_0^t\left(\mathcal{Q}^{+}[u]-\mathcal{Q}^{+}[l]\right)(s)\,ds,\label{equation satisfied by w}
\end{align}
where for the last inequality, we used
\begin{align*}
u\mathcal{R}[l]-l\mathcal{R}[u]&=w\mathcal{R}[l]+l(\mathcal{R}[u]-\mathcal{R}[l]) \geq 0
\end{align*}
since used $w\geq 0.$
We now use \eqref{aux bound} followed by \eqref{comparison u, l}, the fact that $u_0=\widetilde{f}$ and \eqref{gain solution estimate}, to bound the right hand side of \eqref{equation satisfied by w} as follows:
\begin{align*}
\Vert \l k \r^{l} w \Vert_{L^r} &\leq \sum_{\mathcal{T}\in\{\mathcal{G}_0,\mathcal{G}_1\}}\int_0^t \big\|\l k\r^l\big(\mathcal{T}(u,u,u)(s)-\mathcal{T}(l,l,l)(s)\big)\big\|_{L^r}\,ds\\
&\leq 24CR^2T \|\l k\r^l w\|_{L^r} = \frac{1}{4} \Vert \l k \r^{l} w \Vert_{L^r} .
\end{align*}
We conclude that $w=0$ so $u=l$.

Defining $f:=l=u$ and using \eqref{integrated loss} or \eqref{integrated gain}, we obtain that $f$ is a solution to the initial value problem \eqref{KWE}. Moreover $f\geq 0$ since $l\geq 0$ by \eqref{comparison u, l}. By  uniqueness, the positivity of the solution follows.

\end{document}